\documentclass[a4paper, 11pt]{article}
\usepackage[a4paper,hmargin={2.54cm,2.54cm},vmargin={3.17cm,3.17cm}]{geometry}

\usepackage{amsmath,amssymb}
\usepackage{amsthm}

\newcommand{\const}{\mathrm{const}}

\newcommand{\Tr}{\mathrm{Tr}}
\newcommand{\Arg}{\mathrm{Arg}}
\newcommand{\Ei}{\mathrm{Ei}}

\renewcommand{\r}{^{ \mathrm{r} }}

\renewcommand{\Im}{\mathrm{Im}}
\renewcommand{\Re}{\mathrm{Re}}

\newtheorem{lemma}{Lemma}
\newtheorem{theorem}{Theorem}
\newtheorem{statement}{Statement}

\numberwithin{equation}{section}
\numberwithin{lemma}{section}
\numberwithin{statement}{section}
\numberwithin{definition}{section}

\begin{document}

\begin{center}
\large \textbf{Absence of sufficiently localized traveling wave solutions for the Novikov-Veselov equation at zero energy}
\end{center}

\begin{center}
A.V. Kazeykina \footnote{Centre des Math\'ematiques Appliqu\'ees, Ecole Polytechnique, Palaiseau, 91128, France; \\ email: kazeykina@cmap.polytechnique.fr}
\end{center}

\textbf{Abstract.} We demonstrate that the Novikov--Veselov equation (a $ ( 2 + 1 ) $-dimensional analog of KdV) at zero energy does not possess solitons with the space localization stronger than $ O( | x |^{ -4 } ) $ as $ | x | \to \infty $.

\section{Introduction}
In this article we consider the Novikov-Veselov equation
\begin{subequations}
\label{NV}
\begin{align}
\label{NV_eq}
& \partial_t v = 4 \Re ( 4 \partial_z^3 v + \partial_z( v w ) - E \partial_{ z } w ), \\
\label{w_def}
& \partial_{ \bar z } w = - 3 \partial_z v, \quad v = \bar v, \text{ i.e. } v \text{ is a real-valued function, } \quad E \in \mathbb{R},\\
\label{misc}
& v = v( x, t ), \quad w = w( x, t ), \quad x = ( x_1, x_2 ) \in \mathbb{R}^2, \quad t \in \mathbb{R},
\end{align}
\end{subequations}
where
\begin{equation}
\label{derivatives}
\partial_t = \frac{ \partial }{ \partial t }, \quad \partial_z = \frac{ 1 }{ 2 } \left( \frac{ \partial }{ \partial x_1 } - i \frac{ \partial }{ \partial x_2 } \right), \quad \partial_{ \bar z } = \frac{ 1 }{ 2 } \left( \frac{ \partial }{ \partial x_1 } + i \frac{ \partial }{ \partial x_2 } \right).
\end{equation}

From the mathematical point of view equation (\ref{NV}) is the most natural $ ( 2 + 1 ) $-dimensional analog of the classic Korteweg-de Vries equation. When $ v = v( x_1, t ) $, $ w = w( x_1, t ) $, equation (\ref{NV}) reduces to KdV. Besides, equation (\ref{NV}) is integrable via the scattering transform for the Schr\"odinger equation on the plane
\begin{equation}
\label{t_schrodinger}
\begin{aligned}
& L \psi = E \psi, \quad L = - \Delta + v( x, t ), \\
& x \in \mathbb{R}^2, \quad \Delta = \frac{ \partial^2 }{ \partial x_1^2 } + \frac{ \partial^2 }{ \partial x_2^2 }, \quad  E = E_{fixed}.
\end{aligned}
\end{equation}
Note also that when $ E \to \pm \infty $, equation (\ref{NV}) transforms into another renowned $ ( 2 + 1 ) $-dimensional analog of KdV, Kadomtsev-Petviashvili equation (KP-I and KP-II, respectively). In addition, a dispersionless analog of (\ref{NV}) at $ E = 0 $ was derived in \cite{KM} in the framework of a geometrical optics model.

Equation (\ref{NV}) is contained implicitly in \cite{M} as an equation possessing the following representation
\begin{equation}
\label{L-A-B}
\frac{ \partial( L - E ) }{ \partial t } = [ L - E, A ] + B( L - E )
\end{equation}
(Manakov $ L - A - B $ triple), where $ L $ is the operator of the corresponding scattering problem, $ A $, $ B $ are some appropriate differential operators and $ [ \cdot, \cdot ] $ denotes the commutator. For the particular case of the $ 2 $-dimensional Schr\"odinger operator as in (\ref{t_schrodinger}) the explicit form of a third-order differential operator $ A $ and a zero-order differential operator $ B $
\begin{equation}
\label{A-B}
\begin{array}{l}
A = - 8 \partial_{ z }^3 - 2 w \partial_{ z } - 8 \partial_{ \bar z }^3 - 2 \bar w \partial_{ \bar z }, \\
B = 2 \partial_{ \bar z } w + 2 \partial_{ \bar z } \bar w,
\end{array} \text{ where } w \text{ is defined via (\ref{w_def})},
\end{equation}
as well as the corresponding evolution equation (\ref{NV}) and its higher-order analogs were given in \cite{NV1}, \cite{NV2}.

In this article we are concerned with equation (\ref{NV}) at zero energy $ E = 0 $. We consider its regular, sufficiently localized solutions satisfying the following conditions
\begin{align}
\label{smoothness}
& \bullet v, w \in C( \mathbb{R}^2 \times \mathbb{R} ), \; v( \cdot, t ) \in C^3( \mathbb{R}^2 ) \quad \forall t \in \mathbb{R}; \\
\label{decrease}
& \bullet | \partial_{ x }^{ j } v( x, t ) | \leqslant \frac{ q( t ) }{ ( 1 + | x | )^{ 4 + \varepsilon } }, \; j = ( j_1, j_2 ) \in ( \mathbb{N} \cup 0 )^2, \; j_1 + j_2 \leqslant 3, \text{ for some } q( t ) > 0, \varepsilon > 0; \\
\label{w_decrease}
& \bullet | w( x, t ) | \to 0, \text{ when } | x | \to \infty, \quad t \in \mathbb{R}.
\end{align}
Our main concern will be the question of existence and absence of solitons for the Novikov-Veselov equation at zero energy. We say that a solution of (\ref{NV}) is a soliton if $ v( x, t ) = V( x - ct ) $ for some $ c = ( c_1, c_2 ) \in \mathbb{R}^2 $.

Solitons and the large time asymptotic behavior of sufficiently localized in space solutions for the Novikov-Veselov equation (\ref{NV}) were studied in the series of works \cite{GN1, G, N2, K1, KN1, KN2, KN3}. In \cite{KN1, K1} it was shown that in the regular case, i.e. when the scattering data are nonsingular at fixed nonzero energy (and for the reflectionless case at positive energy), related solutions of (\ref{NV}) do not contain isolated solitons in the large time asymptotics. In the general case it was shown in \cite{N2, KN3} that the Novikov-Veselov equation (\ref{NV}) at nonzero energy does not admit exponentially localized solitons. This result was improved in \cite{K3} where it was shown that the Novikov-Veselov equation (\ref{NV}) at nonzero energy does not possess solitons decaying as $ O\left( | x |^{ - 3 - \varepsilon } \right) $, $ \varepsilon > 0 $, for $ | x | \to \infty $. A family of algebraically localized solitons for the Novikov-Veselov equation at positive energy was constructed in \cite{G} (see also discussion in \cite{KN2}). These solitons are rational functions decaying as $ O\left( | x |^{ - 2 } \right) $ when $ | x | \to \infty $. For the Novikov-Veselov equation at zero energy the absence of solitons of conductivity type was proved in \cite{K2}.

The main result of the present article consists in the following theorem.
\begin{theorem}
\label{main_theorem}
Let $ (v, w) $ be a soliton solution of (\ref{NV}) with $ E = 0 $ satisfying properties (\ref{smoothness})-(\ref{w_decrease}). Then $ v \equiv 0 $, $ w \equiv 0 $.
\end{theorem}
The proof is based on the ideas proposed in \cite{N2, K2} and exploits the regularization of Faddeev solutions and some special scattering data introduced in \cite{BLMP1}.

Note that KP-I equation possesses soliton solutions and these solutions decay as $ O\left( | x |^{ - 2 } \right) $ when $ | x | \to \infty $. By contrast, KP-II does not possess sufficiently localized soliton solutions. For the results on existence and nonexistence of localized soliton solutions of KP-I, KP-II and their generalized versions see \cite{BS1}; the symmetry properties and the decay rates of these solutions were derived in \cite{BS2}.

For more results on integrable $ ( 2 + 1 ) $-dimensional systems admitting localized soliton solutions, see \cite{AC,BLMP2,FA,FS} and references therein. In particular, for more results on the behavior of solutions for the Novikov-Veselov equation at zero energy see \cite{LMS,TT, Ch, P}.

The present paper is organized as follows. In Section \ref{schr_section} we recall, in particular, some known notions and results from the direct and inverse scattering theory for the two-dimensional Schr\"odinger equation at zero energy (see \cite{GN2,N1,BLMP1} for more details). Our main result (namely, Theorem \ref{main_theorem}) is proved in Section \ref{main_section}. Section \ref{lemmas_section} contains the proofs of some preliminary lemmas formulated in Section \ref{main_section}. A detailed derivation of some of the formulas used in the paper and proofs of some auxiliary statements can be found in the Appendix.

This work was fulfilled in the framework of research carried out under the supervision of R.G. Novikov.

\section{Direct scattering for the $ 2 $-dimensional Schr\"odinger equation at zero energy}
\label{schr_section}
Consider the $ 2 $-dimensional Schr\"odinger equation at zero energy
\begin{equation}
\label{schrodinger}
\begin{aligned}
& L \psi = 0, \quad L = - \Delta + v, \quad \Delta = 4 \partial_z \partial_{ \bar z }, \\
& v = v( z ), \quad z = x_1 + i x_2, \quad x_1, x_2 \in \mathbb{R}
\end{aligned}
\end{equation}
with a potential $ v $ satisfying the following conditions
\begin{equation}
\label{v_conditions}
\begin{aligned}
& v( z ) = \overline{ v( z ) }, \quad v( z ) \in L^{ \infty }( \mathbb{C} ), \\
& | \partial_{ z }^{ j_1 } \partial_{ \bar z }^{ j_2 } v( z ) | < q ( 1 + | z | )^{ - 4 - \varepsilon } \text{ for some } q > 0, \; \varepsilon > 0, \text{ where } j_1 + j_2 \leqslant 3.
\end{aligned}
\end{equation}

For $ \lambda \in \mathbb{C} $  we consider solutions $ \psi_1( z, \lambda ) $, $ \psi_2( z, \lambda ) $, $ \psi_3( z, \lambda ) $ of (\ref{schrodinger}) having the following asymptotics
\begin{align}
\label{psi1_asympt}
& \psi_1( z, \lambda ) = e^{ i \lambda z } \mu_1( z, \lambda ), \quad \mu_1( z, \lambda ) = 1 + \overline{o}( 1 ), \text{ as } | z | \to \infty, \\
\label{psi2_asympt}
& \psi_2( z, \lambda ) = e^{ i \lambda z } \mu_2( z, \lambda ), \quad \mu_2( z, \lambda ) = z + \overline{o}( 1 ), \text{ as } | z | \to \infty, \\
\label{psi3_asympt}
& \psi_3( z, \lambda ) = e^{ i \lambda z } \mu_3( z, \lambda ), \quad \mu_3( z, \lambda ) = z^2 + \overline{o}( 1 ), \text { as } | z | \to \infty.
\end{align}

Solutions of (\ref{schrodinger}) with asymptotics (\ref{psi1_asympt}) are known in literature as Faddeev's exponentially growing solutions (see \cite{F,BLMP1}). Solutions of (\ref{schrodinger}) with asymptotics (\ref{psi2_asympt}) were introduced in \cite{BLMP1}.

Functions $ \mu_1( z, \lambda ) $, $ \mu_2( z, \lambda ) $, $ \mu_3( z, \lambda ) $ defined by (\ref{psi1_asympt})-(\ref{psi3_asympt}) can be also represented as solutions of the following integral equations
\begin{align}
\label{mu_int_equation}
& \mu_1( z, \lambda ) = 1 + \iint\limits_{ \mathbb{C} } g( z - \xi, \lambda ) v( \xi ) \mu_1( \xi, \lambda ) d \Re \xi d \Im \xi, \\
\label{nu_int_equation}
& \mu_2( z, \lambda ) = z + \iint\limits_{ \mathbb{C} } g( z - \xi, \lambda ) v( \xi ) \mu_2( \xi, \lambda ) d \Re \xi d \Im \xi, \\
\label{ksi_int_equation}
& \mu_3( z, \lambda ) = z^2 + \iint\limits_{ \mathbb{C} } g( z - \xi, \lambda ) v( \xi ) \mu_3( \xi, \lambda ) d \Re \xi d \Im \xi, \text{ where } \\
\label{green}
& g( z, \lambda ) = - \left( \frac{ 1 }{ 4 \pi } \right)^2 \iint\limits_{ \mathbb{C} } \frac{ e^{ \frac{ i }{ 2 } ( p \bar z + \bar p z ) } }{ p \bar p + 2 p \lambda } d \Re p  d \Im p,
\end{align}
where $ z \in \mathbb{C} $, $ \lambda \in \mathbb{C} \backslash 0 $.

The integral in (\ref{green}) can be computed explicitly (see \cite{BLMP1}):
\begin{equation*}
g( z, \lambda ) = \frac{ 1 }{ 16 \pi } \exp( - i \lambda z ) \left( \Ei( i \lambda z ) + \Ei( - i \bar \lambda \bar z ) \right).
\end{equation*}
Here $ \Ei( z ) $ is the exponential-integral function defined as follows
\begin{equation}
\label{ei_definition}
\Ei( z ) - \gamma - \ln( -z ) = \sum_{ n = 1 }^{ \infty } \frac{ z^n }{ n n! }, \quad z \in \mathbb{C} \backslash ( \mathbb{R}_+ \cup 0 ),
\end{equation}
where $ \gamma $ is the Euler-Mascheroni constant
\begin{equation*}
\gamma = - \int\limits_{ 0 }^{ \infty } e^{ -x } \ln x dx,
\end{equation*}
the branch cut for the logarithm function is taken on the negative real axis: $ \ln( z ) = \ln| z | + i \Arg z $, $ | \Arg z | < \pi $, and the series in the right part of (\ref{ei_definition}) converges on the whole complex plane.

\begin{statement}
\label{ei_statement}
Function $ \Ei( z ) $ defined by (\ref{ei_definition}) possesses the following properties:
\begin{enumerate}
\item \label{ei_cont} $ \Ei( z ) + \Ei( \bar z ) $ can be extended uniquely and continuously to $ z \in \mathbb{R}_+ $ via
\begin{equation*}
\Ei( z ) + \Ei( \bar z ) = 2 \gamma + 2 \ln | z | + \sum_{ n = 1 }^{ \infty } \frac{ z^n + \bar z^n }{ n n! };
\end{equation*}
\item \label{ei_sym} $ \overline{ \Ei( z ) } = \Ei( \bar z ) $;
\item $ | \Ei( z ) + \Ei( \bar z ) | \leqslant C_{ \delta } \ln( | z |^2 ) $ for $ 0 < | z | \leqslant \delta $ (function $ \Ei( z ) + \Ei( \bar z ) $ has an integrable singularity at $ z = 0 $);
\item \label{ei_der} $ \partial_z \Ei( z ) = \dfrac{ e^z }{ z } $ for $ z \in \mathbb{C} \backslash ( \mathbb{R}_+ \cup 0 ) $;
\item \label{ei_int} $ \Ei( z ) $ possesses the following representation $ \Ei( z ) = \int\limits_{ \Gamma_z } \frac{ e^{ \tau } }{ \tau } d \tau = \int\limits_{ -\infty }^{ z } \frac{ e^{ \tau } }{ \tau } d \tau $, where $ \Gamma_z $ is any contour on the cut complex plane $ \mathbb{C} \backslash ( \mathbb{R}_+ \cup 0 ) $ connecting points $ -\infty $ and $ z $;
\item \label{ei_est} $ | e^{ -z }( \Ei( z ) + \Ei( \bar z ) ) | \leqslant \frac{ C }{ | z | } $ for $ | z | > 0 $.
\end{enumerate}
\end{statement}
Properties \ref{ei_cont}-\ref{ei_int} are well-known in literature (see, for example, \cite{BLMP1} for some of them) and can be easily derived from (\ref{ei_definition}). A detailed proof of property \ref{ei_est} is presented in subsection \ref{ei_subs} of Appendix.

\bigskip
It is easy to see that function $ g( z, \lambda ) $ has a logarithmic singularity at $ \lambda = 0 $ and thus the functions $ \mu_1( z, \lambda ) $, $ \mu_2( z, \lambda ) $, $ \mu_3( z, \lambda ) $ are not generally defined for $ \lambda = 0 $ even for arbitrarily small values of $ v $. In \cite{BLMP1} the following regularization of (\ref{green}) at $ \lambda = 0 $ was proposed:
\begin{equation}
\label{regg_def}
g\r( z, \lambda ) = \frac{ 1 }{ 16 \pi } \exp( -i \lambda z ) ( \Ei( i \lambda z ) + \Ei( -i \bar \lambda \bar z ) ) - \frac{ 1 }{ 16 \pi } ( 1 + X( z, \lambda ) ) \mathcal{G}( \lambda )
\end{equation}
where
\begin{equation}
\label{X_def}
X( z, \lambda ) = \exp( - i \lambda z - i \bar \lambda \bar z ), \quad \mathcal{G}( \lambda ) = \frac{ 1 }{ 4 } ( \exp( -i \lambda ) + \exp( i \bar \lambda ) ) \left( \Ei( i \lambda ) + \Ei( -i \bar \lambda ) \right).
\end{equation}

\begin{statement}
\label{g_statement}
Function $ g\r( z, \lambda ) $ defined by (\ref{regg_def}), (\ref{ei_definition}) possesses the following properties:
\begin{enumerate}
\item \label{g_zero} $ g\r( z, 0 ) = \frac{ 1 }{ 16 \pi } \ln | z |^2 $;
\item \label{g_ext} $ g\r( z, \lambda ) $ can be uniquely and continuously defined for $ z \in \mathbb{C} \backslash 0 $, $ \lambda \in \mathbb{C} $ via
\begin{multline*}
g\r( z, \lambda ) = \frac{ 1 }{ 16 \pi } \exp( - i \lambda z ) \left\{ 2 \gamma + 2 \ln | \lambda z | + \sum_{ n = 1 }^{ \infty } \frac{ ( i \lambda z )^ n + ( - i \bar \lambda \bar z )^n }{ n n! } \right\} - \\
- \frac{ 1 }{ 64 \pi } ( 1 + X( z, \lambda ) )( \exp( - i \lambda ) + \exp( i \bar \lambda ) ) \left\{ 2 \gamma + 2 \ln | \lambda | + \sum_{ n = 1 }^{ \infty } \frac{ ( i \lambda )^n + ( - i \bar \lambda )^n }{ n n! } \right\} \text{ for } \lambda \neq 0
\end{multline*}
and via the previous item of the current statement for $ \lambda = 0 $;
\item \label{g_sym} $ \overline{ g\r( z, \lambda ) } X( z, \lambda ) =  g\r( z, \lambda ) $;
\item \label{g_decr} $ | g\r( z, \lambda ) | \leqslant \frac{ \const }{ | \lambda | } \left( 1 + \frac{ 1 }{ | z | } \right) $ for $ z \in \mathbb{C} \backslash 0 $, $ \lambda \in \mathbb{C} \backslash 0 $;
\item \label{g_lim} $ g\r( z, \lambda ) = - \frac{ 1 }{ 16 \pi } ( 1 + X( z, \lambda ) ) \mathcal{G}( \lambda ) + \underline{O}\left( \frac{ 1 }{ | z | } \right) $ as $ | z | \to \infty $ for any fixed $ \lambda \in \mathbb{C} \backslash 0 $;
\item \label{g_der} $ \dfrac{ \partial g\r }{ \partial \bar \lambda } = \dfrac{ 1 }{ 16 \pi \bar \lambda } X + \dfrac{ i \bar z }{ 16 \pi } X \mathcal{G} - \dfrac{ 1 }{ 16 \pi } ( 1 + X ) \dfrac{ \partial \mathcal{G} }{ \partial \bar \lambda } $.
\end{enumerate}
\end{statement}
Properties of function $ g\r( z, \lambda ) $ follow directly from definitions (\ref{regg_def}), (\ref{ei_definition}) and Statement \ref{ei_statement}.

Now we define functions $  \mu\r_1( z, \lambda ) $, $ \mu\r_2( z, \lambda ) $, $ \mu\r_3( z, \lambda ) $ as the solutions of the following integral equations:
\begin{align}
\label{tilde_mu1}
& \mu\r_1( z, \lambda ) = 1 + \iint\limits_{ \mathbb{C} } g\r( z - \xi, \lambda ) v( \xi )  \mu\r_1( \xi, \lambda ) d \Re \xi d \Im \xi, \\
\label{tilde_mu2}
& \mu\r_2( z, \lambda ) = z + \iint\limits_{ \mathbb{C} } g\r( z - \xi, \lambda ) v( \xi ) \mu\r_2( \xi, \lambda ) d \Re \xi d \Im \xi, \\
\label{tilde_mu3}
& \mu\r_3( z, \lambda ) = z^2 + \iint\limits_{ \mathbb{C} } g\r( z - \xi, \lambda ) v( \xi ) \mu\r_3( \xi, \lambda ) d \Re \xi d \Im \xi.
\end{align}
We also define the following functions: $ \psi\r_1( z, \lambda ) = e^{ i \lambda z } \mu\r_1( z, \lambda ) $, $ \psi\r_2( z, \lambda ) = e^{ i \lambda z } \mu\r_2( z, \lambda ) $, $ \psi\r_3( z, \lambda ) = e^{ i \lambda z } \mu\r_3( z, \lambda ) $. Note that functions $ \psi\r_1( z, \lambda ) $, $ \psi\r_2( z, \lambda ) $, $ \psi\r_3( z, \lambda ) $ are new ``scattering'' solutions of the Schr\"odinger equation (\ref{schrodinger}).

In terms of  $ m\r_1( z, \lambda ) = ( 1 + | z | )^{ -( 3 + \varepsilon / 2 ) } \mu\r_1( z, \lambda ) $, $ m\r_2( z, \lambda ) = ( 1 + | z | )^{ -( 3 + \varepsilon / 2 ) } \mu\r_2( z, \lambda ) $, $ m\r_3( z, \lambda ) = ( 1 + | z | )^{ - ( 3 + \varepsilon / 2 ) } \mu\r_3( z, \lambda ) $ equations (\ref{tilde_mu1}), (\ref{tilde_mu2}), (\ref{tilde_mu3}) respectively take the forms
\begin{align}
\label{m1_equation}
& m\r_1( z, \lambda ) = ( 1 + | z | )^{ -( 3 + \varepsilon / 2 ) } + \iint\limits_{ \mathbb{C} } ( 1 + | z | )^{ -( 3 + \varepsilon / 2 ) } g\r( z - \xi, \lambda ) \frac{ v( \xi ) }{ ( 1 + | \xi | )^{ -( 3 + \varepsilon / 2 ) } } m\r_1( \xi, \lambda ) d \Re \xi d \Im \xi, \\
\label{m2_equation}
& m\r_2( z, \lambda ) = z ( 1 + | z | )^{ -( 3 + \varepsilon / 2 ) } + \iint\limits_{ \mathbb{C} } ( 1 + | z | )^{ -( 3 + \varepsilon / 2 ) } g\r( z - \xi, \lambda ) \frac{ v( \xi ) }{ ( 1 + | \xi | )^{ -( 3 + \varepsilon / 2 ) } } m\r_2( \xi, \lambda ) d \Re \xi d \Im \xi, \\
\label{m3_equation}
& m\r_3( z, \lambda ) = z^2 ( 1 + | z | )^{ -( 3 + \varepsilon / 2 ) } + \iint\limits_{ \mathbb{C} } ( 1 + | z | )^{ -( 3 + \varepsilon / 2 ) } g\r( z - \xi, \lambda ) \frac{ v( \xi ) }{ ( 1 + | \xi | )^{ -( 3 + \varepsilon / 2 ) } } m\r_3( \xi, \lambda ) d \Re \xi d \Im \xi.
\end{align}
The integral operator $ H\r( \lambda ) $ of the integral equations (\ref{m1_equation}), (\ref{m2_equation}), (\ref{m3_equation}) is a Hilbert-Schmidt operator: more precisely, $ H\r( \cdot, \cdot, \lambda ) \in L^2( \mathbb{C} \times \mathbb{C} ) $, where $ H\r( z, \xi, \lambda ) $ is the Schwartz kernel of the integral operator $ H\r( \lambda ) $, and $ | \Tr ( H\r )^2( \lambda ) | < \infty $. Thus, the modified Fredholm determinant for (\ref{m1_equation}), (\ref{m2_equation}), (\ref{m3_equation}) can be defined by means of the formula:
\begin{equation}
\label{fred_determinant}
\ln \Delta\r( \lambda ) = \Tr( \ln( I - H\r( \lambda ) ) + H\r( \lambda ) ).
\end{equation}
For the precise sense of this definition see \cite{GK}.

We will also define
\begin{equation*}
\mathcal{E}\r = \{ \lambda \in \mathbb{C} \colon \Delta\r( \lambda ) = 0 \}.
\end{equation*}
In this notation $ \mathcal{E}\r $ represents the set of $ \lambda \in \mathbb{C} $ for which either existence or uniqueness of solutions of (\ref{m1_equation}), (or, similarly, of (\ref{m2_equation}) or of (\ref{m3_equation})) fails.

For $ \lambda \in \mathbb{C} \backslash \mathcal{E}\r $ we define the following ``scattering data'' $ S\r( \lambda ) $ for the potential $ v $:
\begin{align}
\label{s_data}
& S\r( \lambda ) = \{ a\r_1( \lambda ), b\r_1( \lambda ), c\r_1( \lambda ), d\r_1( \lambda ), a\r_2( \lambda ), c\r_2( \lambda ), a\r_3( \lambda ) \}, \\
\label{a1_data}
& a\r_1( \lambda ) = \iint\limits_{ \mathbb{C} } v( z ) \mu\r_1( z, \lambda ) d \Re z d \Im z, \\
\label{b1_data}
& b\r_1( \lambda ) = \iint\limits_{ \mathbb{C} } e^{ i \lambda z + i \bar \lambda \bar z } v( z ) \mu\r_1( z, \lambda ) d \Re z d \Im z, \\
\label{c1_data}
& c\r_1( \lambda ) = \iint\limits_{ \mathbb{C} } z v( z ) \mu\r_1( z, \lambda ) d \Re z d \Im z, \\
\label{d1_data}
& d\r_1( \lambda ) = \iint\limits_{ \mathbb{C} } \bar z e^{ i \lambda z + i \bar \lambda \bar z } v( z ) \mu\r_1( z, \lambda ) d \Re z d \Im z, \\
\label{a2_data}
& a\r_2( \lambda ) = \iint\limits_{ \mathbb{C} }  v( z ) \mu\r_2( z, \lambda ) d \Re z d \Im z, \\
\label{c2_data}
& c\r_2( \lambda ) = \iint\limits_{ \mathbb{C} } z v( z ) \mu\r_2( z, \lambda ) d \Re z d \Im z, \\
\label{a3_data}
& a\r_3( \lambda ) = \iint\limits_{ \mathbb{C} } v( z ) \mu\r_3( z, \lambda ) d \Re z d \Im z.
\end{align}

Functions $ a\r_1 $, $ b\r_1 $ are regularized analogs of the standard Faddeev generalized scattering data for the $ 2 $-dimensional Schr\"odinger equation at zero energy. The scattering data $ d\r_1 $, $ a\r_2 $ for the case of the Schr\"odinger equation at zero energy were introduced in \cite{BLMP1}.

The following properties of function $ \Delta\r( \lambda ) $ will play a substantial role in the proof of the main result.
\begin{statement}
\label{delta_prop}
Let $ v $ satisfy conditions (\ref{v_conditions}). Then function $ \Delta\r( \lambda ) $ satisfies the following properties:
\begin{enumerate}
\item \label{delta_continuity} $ \Delta\r \in C( \mathbb{C} ) $;
\item \label{delta_limit} $ \Delta\r( \lambda ) \to 1 $ as $ | \lambda | \to \infty $;
\item \label{real_valued} $ \Delta\r $ is real-valued;
\item $ \Delta\r( \lambda ) $ satisfies the following $ \bar \partial $-equation
\begin{multline}
\label{dif_det}
\frac{ \partial \Delta\r }{ \partial \bar \lambda } = \Biggl\{ \frac{ 1 }{ 16 \pi \bar \lambda } ( - \overline{ a\r_1( \lambda ) } + \hat v( 0 ) ) - \frac{ 1 }{ 16 \pi } \frac{ \partial \mathcal{G} }{ \partial \bar \lambda } ( - \overline{ a\r_1( \lambda ) } - a\r_1( \lambda ) + 2 \hat v( 0 ) )  + \\
+ \frac{ i }{ 16 \pi } \mathcal{G}( \lambda ) ( - \overline{ a\r_2( \lambda ) } + \overline{ c\r_1( \lambda ) } ) \Biggr\} \Delta\r,
\end{multline}
where $ \hat v( 0 ) = \iint\limits_{ \mathbb{C} } v( z ) d \Re z d \Im z $, $ \lambda \in \mathbb{C} \backslash ( \mathcal{E}\r \cup 0 ) $.
\end{enumerate}
\end{statement}
Analogs of these properties of $ \Delta\r $ for the nonregularized determinant in the case of nonzero energy can be found in \cite{HN,N2,KN3}. The properties of $ \Delta\r $ at zero energy can be proved similarly. In particular, property \ref{delta_continuity} is a consequence of continuous dependency of $ H\r( \lambda ) $ on $ \lambda $ (in particular, see Statement \ref{conth_statement} of Appendix for the proof of continuity of $ H\r( \lambda ) $ at $ \lambda = 0 $). Property \ref{delta_limit} follows from items \ref{g_ext}, \ref{g_decr} of Statement \ref{g_statement} (see Statement \ref{vanh_statement} of Appendix for details). Property \ref{real_valued} is a consequence of item \ref{g_sym} of Statement \ref{g_statement}. The derivation of equation (\ref{dif_det}) is based on the ideas proposed in \cite{HN} and is presented in subsection \ref{dif_subs} of Appendix.

\begin{statement}
\label{scat_statement}
Let $ v $ satisfy conditions (\ref{v_conditions}). Then
\begin{enumerate}
\item \label{mu_continuity} $ \mu\r_1( z, \lambda ) $ is a continuous function of $ \lambda $ on $ \mathbb{C} \backslash \mathcal{E}\r $;
\item \label{mu_dbar_prop} $ \mu\r_1( \lambda ) $ satisfies the following $ \bar \partial $-equation:
\begin{equation}
\label{mu_dbar}
\frac{ \partial \mu\r_1 }{ \partial \bar \lambda } = \frac{ 1 }{ 16 \pi }\left\{ \frac{ 1 }{ \bar \lambda } X b\r_1 - \frac{ \partial \mathcal{G} }{ \partial \bar \lambda } X b\r_1 - i \mathcal{G} X d\r_1 \right\} \overline{ \mu\r_1 } - \frac{ 1 }{ 16 \pi } \frac{ \partial \mathcal{G} }{ \partial \bar \lambda } a\r_1 \mu\r_1 + \frac{ i }{ 16 \pi } \mathcal{G} X b\r_1 \overline{ \mu\r_2 } ,\\
\end{equation}
for $ \lambda \in \mathbb{C} \backslash ( 0 \cup \mathcal{E}\r ) $;
\item \label{mu_limits} $ \mu\r_1 \to 1 $, as $ \lambda \to \infty $;
\item \label{scat_cont} the scattering data $ S\r( \lambda ) $ of (\ref{s_data}) are continuous on $ \mathbb{C} \backslash \mathcal{E}\r $;
\item \label{a_limit} $ \hat v( 0 ) = \lim\limits_{ \lambda \to \infty } a\r_1( \lambda ) $, where $ \hat v( 0 ) = \iint\limits_{ \mathbb{C} } v( z ) d \Re z d \Im z $.
\end{enumerate}
\end{statement}
Items \ref{mu_continuity}, \ref{scat_cont}, are a consequence of continuous dependency of $ H\r( \lambda ) $ on $ \lambda $ (in particular, see Statement \ref{conth_statement} of Appendix for the proof of continuity of $ H\r( \lambda ) $ at $ \lambda = 0 $). Items \ref{mu_dbar_prop}, \ref{mu_limits} were proved in \cite{BLMP1}. Item \ref{a_limit} follows from items \ref{mu_limits}, \ref{scat_cont}.

\section{Proof of Theorem \ref{main_theorem}}
\label{main_section}
We will start this section by formulating some preliminary lemmas. The proofs of these lemmas are given in Section \ref{lemmas_section}.

\begin{lemma}
\label{shift_lemma}
Let $ v( z ) $ be a potential satisfying (\ref{v_conditions}) with the modified Fredholm determinant $ \Delta\r( \lambda ) $ defined by (\ref{fred_determinant}) in the framework of equation (\ref{schrodinger}) and the scattering data
\begin{equation}
\mathcal{ S }\r( \lambda ) = \{ a\r_1( \lambda ), b\r_1( \lambda ), c\r_1( \lambda ), d\r_1( \lambda ), a\r_2( \lambda ), c\r_2( \lambda ), a\r_3( \lambda ) \} , \quad \lambda \in \mathbb{C} \backslash \mathcal{E}\r,
\end{equation}
defined by (\ref{a1_data})-(\ref{a3_data}) in the framework of equation (\ref{schrodinger}). Then the modified Fredholm determinant $ \Delta\r_{ \zeta }( \lambda ) $ and the scattering data $ \mathcal{ S }\r_{ \zeta }( \lambda ) $,
\begin{equation*}
\mathcal{S}\r_{ \zeta }( \lambda ) = \{ a\r_{ 1, \zeta }( \lambda ), b\r_{ 1, \zeta }( \lambda ), c\r_{ 1, \zeta }( \lambda ), d\r_{ 1, \zeta }( \lambda ), a\r_{ 2, \zeta }( \lambda ), c\r_{ 2, \zeta }( \lambda ), a\r_{ 3, \zeta }( \lambda ) \},
\end{equation*}
for the potential $ v_{ \zeta }( z ) = v( z - \zeta ) $ have the following properties:
\begin{enumerate}
\item \label{det_shift} $ \Delta\r_{ \zeta }( \lambda ) = \Delta\r( \lambda ) $;

\item scattering data $ S\r_{ \zeta }( \lambda ) $ are defined for $ \lambda \in \mathbb{C} \backslash \mathcal{E}\r $ and are related to $ \mathcal{ S }\r( \lambda ) $ by the following formulas
\begin{align}
\label{a1_zeta}
& a\r_{ 1, \zeta }( \lambda ) = a\r_1( \lambda ), \\
\label{b1_zeta}
& b\r_{ 1, \zeta }( \lambda ) =  e^{ i \lambda \zeta + i \bar \lambda \bar \zeta } b\r_1( \lambda ), \\
\label{c1_zeta}
& c\r_{ 1, \zeta }( \lambda ) = c\r_1( \lambda ) + \zeta a\r_1( \lambda ), \\
\label{d1_zeta}
& d\r_{ 1, \zeta }( \lambda ) = e^{ i \lambda \zeta + i \bar \lambda \bar \zeta }( d\r_1( \lambda ) + \bar \zeta b\r_1( \lambda ) ), \\
\label{a2_zeta}
& a\r_{ 2, \zeta }( \lambda ) = a\r_2( \lambda ) + \zeta a\r_1( \lambda ), \\
\label{c2_zeta}
& c\r_{ 2, \zeta }( \lambda ) = c\r_2( \lambda ) + \zeta ( a\r_2( \lambda ) + c\r_1( \lambda ) ) + \zeta^2 a\r_1( \lambda ), \\
\label{a3_zeta}
& a\r_{ 3, \zeta }( \lambda ) = a\r_3( \lambda ) + 2 \zeta a\r_2( \lambda ) + \zeta^2 a\r_1( \lambda ).
\end{align}

\end{enumerate}

\end{lemma}

\begin{lemma}
\label{dyn_lemma}
Let $ ( v, w ) $ satisfy equation (\ref{NV}) and conditions (\ref{smoothness})-(\ref{w_decrease}). Let $ S\r( \lambda, t ) $ be the scattering data for $ v $ defined by (\ref{a1_data})-(\ref{a3_data}) for a certain $ \lambda \in \mathbb{C} \backslash ( \mathcal{E}\r \cup 0 ) $ and all $ t \in \mathbb{R} $ in the framework of equation (\ref{schrodinger}). Then the evolution of these scattering data is described as follows:
\begin{align}\
\label{a1_evolution}
& a\r_1( \lambda, t ) = a\r_1( \lambda, 0 ), \\
\label{b1_evolution}
& b\r_1( \lambda, t ) = e^{ 8 i ( \lambda^3 + \bar \lambda^3 ) t } b\r_1( \lambda, 0 ), \\
\label{c1_evolution}
& c\r_1( \lambda, t ) = c\r_1( \lambda, 0 ) + 24 \lambda^2 a\r_1( \lambda, 0 ) t, \\
\label{d1_evolution}
& d\r_1( \lambda, t ) = e^{ 8 i ( \lambda^3 + \bar \lambda^3 ) t } \left( d\r_1( \lambda, 0 ) + 24 \bar \lambda^2 b\r_1( \lambda, 0 ) t \right)\\
\label{a2_evolution}
& a\r_2( \lambda, t ) = a\r_2( \lambda, 0 ) + 24 \lambda^2 a\r_1( \lambda, 0 ) t, \\
\label{c2_evolution}
& c\r_2( \lambda, t ) = c\r_2( \lambda, 0 ) + 24 \lambda^2 ( a\r_2( \lambda, 0 ) + c\r_1( \lambda, 0 ) ) t + ( 24 \lambda^2 )^2 a\r_1( \lambda, 0 ) t^2, \\
\label{a3_evolution}
& a\r_3( \lambda, t ) = a\r_3( \lambda, 0 ) + 48 \lambda^2 a\r_2( \lambda, 0 ) t - 48 i \lambda a\r_1( \lambda, 0 ) t + ( 24 \lambda^2 )^2 a\r_1( \lambda, 0 ) t^2.
\end{align}
\end{lemma}

The remaining part of the proof of Theorem \ref{main_theorem} consists in the following. First of all we note that since $ ( v, w ) $ is a soliton, from Lemma \ref{shift_lemma} it follows that the set $ \mathcal{E}\r $ of values of $ \lambda \in \mathbb{C} $ for which the scattering data $ S\r( \lambda ) $ are not well-defined does not depend on $ t $.

Since $ ( v, w ) $ is a soliton, the time dynamics of its scattering data $ b\r_1 $ is described by the formula
\begin{equation*}
b\r_1( \lambda, t ) = \exp( i ( \lambda c + \bar \lambda \bar c ) t ) b\r_1( \lambda, 0 ).
\end{equation*}
(see formula (\ref{b1_zeta}) of Lemma \ref{shift_lemma}). Combining this with (\ref{b1_evolution}) from Lemma \ref{dyn_lemma} gives
\begin{equation*}
\exp( 8 i ( \lambda^3 + \bar \lambda^3 ) t ) b\r_1( \lambda, 0 ) = \exp( i ( \lambda c + \bar \lambda \bar c ) t ) b\r_1( \lambda, 0 ).
\end{equation*}
Since functions $ \lambda $, $ \bar \lambda $, $ \lambda^3 $, $ \bar \lambda^3 $, $ 1 $ are linearly independent in any neighborhood of any point in $ \mathbb{C} $ and $ b\r_1( \lambda, 0 ) $ is continuous on $ \mathbb{C} \backslash ( \mathcal{E}\r \cup 0 ) $, we obtain that
\begin{equation}
\label{b_null}
b\r_1( \lambda, t ) \equiv 0 \text{ for } \lambda \in \mathbb{C} \backslash ( \mathcal{E}\r \cup 0 ).
\end{equation}
Similarly, combining formulas (\ref{d1_zeta}) with $ \zeta = ct $, (\ref{d1_evolution}) and property (\ref{b_null}) we obtain that
\begin{equation}
\label{d_null}
d\r_1( \lambda, t ) \equiv 0 \text{ for } \lambda \in \mathbb{C} \backslash ( \mathcal{E}\r \cup 0 ).
\end{equation}

From (\ref{a2_zeta}), (\ref{a2_evolution}) we get that
\begin{equation*}
a\r_2( \lambda, 0 ) + c t a\r_1( \lambda, 0 ) = a\r_2( \lambda, 0 ) + 24 \lambda^2 t a\r_1( \lambda, 0 ).
\end{equation*}
Thus $ a\r_1( \lambda, 0 ) \equiv 0 $ on $ \mathbb{C} \backslash \{ \mathcal{E}\r \cup 0 \} $ and formula (\ref{a1_evolution}) implies that
\begin{equation}
\label{a_null}
a\r_1( \lambda, t ) \equiv 0 \text{ for } \lambda \in \mathbb{C} \backslash ( \mathcal{E}\r \cup 0 )
\end{equation}
and $ a\r_2( \lambda, t ) \equiv a\r_2( \lambda, 0 ) $, $ c\r_1( \lambda, t ) \equiv c\r_1( \lambda, 0 ) $ for $ \lambda \in \mathbb{C} \backslash \{ \mathcal{E}\r \cup 0 \} $. From item \ref{a_limit} of Statement \ref{scat_statement} it follows also that
\begin{equation}
\label{hatv_null}
\hat v( 0 )  = 0.
\end{equation}

Combining (\ref{c2_zeta}) with (\ref{c2_evolution}) and (\ref{a_null}) we derive that
\begin{equation}
\label{alpha_eq_c}
a\r_2( \lambda, t ) \equiv - c\r_1( \lambda, t ) \text{ for } \lambda \in \mathbb{C} \backslash ( \mathcal{E}\r \cup 0 )
\end{equation}
and combing (\ref{a3_zeta}) with (\ref{a3_evolution}), (\ref{a_null}), (\ref{alpha_eq_c}) yields
\begin{equation}
\label{alphac_null}
a\r_2( \lambda, t ) \equiv 0, \quad c\r_1( \lambda, t ) \equiv 0 \text{ for } \lambda \in \mathbb{C} \backslash ( \mathcal{E}\r \cup 0 ).
\end{equation}

Now equation (\ref{dif_det}) together with properties (\ref{a_null}), (\ref{hatv_null}) (\ref{alphac_null}) implies that $ \Delta\r $ is holomorphic on $ \mathbb{C} \backslash ( \mathcal{E}\r \cup 0 ) $.

Suppose that $ \mathcal{E}\r \neq \varnothing $. Item \ref{delta_limit} of Statement \ref{delta_prop} implies that $ \mathcal{E}\r $ is bounded. Since $ \mathcal{E}\r $ is a closed set, then there exists $ \lambda_* \in \mathcal{E}\r $ such that $ | \lambda_* | = \max\limits_{ \lambda \in \mathcal{E}\r } | \lambda | $. Function $ \Delta\r( \lambda ) $ is holomorphic in $ D = \{ \lambda \in \mathbb{C} \colon | \lambda | > | \lambda_* | \} $. Items \ref{delta_limit}, \ref{real_valued} of Statement \ref{delta_prop} together with holomorphicity of $ \Delta\r $ in $ D $ imply that $ \Delta\r( \lambda ) \equiv 1 $ on $ \lambda \in D $. On the other hand, $ \Delta\r( \lambda_* ) = 0 $, which contradicts property \ref{delta_continuity} of Statement \ref{delta_prop}. Thus we obtain that $ \mathcal{E}\r = \varnothing $ and $ \Delta\r \equiv 1 $ on $ \mathbb{C} $.

The function $ \mu\r_1 $ is holomorphic on $ \mathbb{C} \backslash 0 $ as follows from (\ref{mu_dbar}), item \ref{mu_continuity} of Statement \ref{scat_statement}, (\ref{b_null}), (\ref{d_null}), (\ref{a_null}) and the established fact that $ \mathcal{E}\r = \varnothing $. The function $ \mu\r_1 $ is also bounded on the whole complex plane due to the items \ref{mu_continuity}, \ref{mu_limits} of Statement \ref{scat_statement}, which implies, in particular, the holomorphicity of function $ \mu\r_1 $ on the whole complex plane $ \mathbb{C} $. From Liouville's theorem it follows that $ \mu\r_1 \equiv 1 $. Then, finally, from (\ref{schrodinger}) with $ \psi( z, \lambda ) = e^{ i \lambda z } $ we obtain that $ v \equiv 0 $.

\section{Proofs of Lemmas \ref{shift_lemma}, \ref{dyn_lemma}}
\label{lemmas_section}

\begin{proof}[Proof of Lemma \ref{shift_lemma}] \rule{1pt}{0pt}

\begin{enumerate}

\item Note that the eigenvalues of the integral operator $ H\r( \lambda ) $ of integral equations (\ref{m1_equation})-(\ref{m3_equation}) with potential $ v( z ) $ coincide with the eigenvalues of the integral operator $ H\r_{ \zeta }( \lambda ) $ of integral equations (\ref{m1_equation})-(\ref{m3_equation}) with potential $ v_{ \zeta }( z ) = v( z - \zeta ) $.

Indeed, from
\begin{equation*}
\iint\limits_{ \mathbb{C} } g\r( z - \xi, \lambda ) \frac{ v( \xi ) ( 1 + | \xi | )^{ 3 + \varepsilon / 2 } }{ ( 1 + | z | )^{ 3 + \varepsilon / 2 } } m( \xi, \lambda ) d \Re \xi d \Im \xi = \nu m( z, \lambda ), \quad \nu \in \mathbb{C},
\end{equation*}
it follows that $ m_{ \zeta }( z, \lambda ) = m( z - \zeta, \lambda )  ( 1 + | z - \zeta | )^{ 3 + \varepsilon / 2 } ( 1 + | z | )^{ - ( 3 + \varepsilon / 2 ) } $ satisfies the following equation
\begin{equation*}
\iint\limits_{ \mathbb{C} } g\r( z - \xi, \lambda ) \frac{ v( \xi - \zeta ) ( 1 + | \xi | )^{ 3 + \varepsilon / 2 } }{ ( 1 + | z | )^{ 3 + \varepsilon / 2 } } m_{ \zeta }( \xi, \lambda ) d \Re \xi d \Im \xi = \nu m_{ \zeta }( z, \lambda ).
\end{equation*}

Since the modified Fredholm determinant of an operator is defined by its eigenvalues uniquely (see \cite{GK}), item \ref{det_shift} of Lemma \ref{shift_lemma} is proved.

Item \ref{det_shift} of Lemma \ref{shift_lemma} can also be proved by considering the right-hand side of (\ref{fred_determinant}) as a sum of convergent series; then it can be checked that every member of this series for $ \Delta\r_{ \zeta }( \lambda ) $ coincides with the corresponding member of the series for $ \Delta\r( \lambda ) $.

\item
Function $ \mu\r_{ 1, \zeta }( z, \lambda ) $ corresponding to the potential $ v_{ \zeta }( z ) = v( z - \zeta ) $ is defined as the solution of the following equation
\begin{equation*}
\mu\r_{ 1, \zeta }( z, \lambda ) = 1 + \iint\limits_{ \mathbb{C} } g\r( z - \xi, \lambda ) v( \xi - \zeta ) \mu\r_{ 1, \zeta }( \xi, \lambda ) d \Re \xi d \Im \xi.
\end{equation*}
The change of variables $ \xi - \zeta = q $, $ z - \zeta = \eta $ yields the following equation on $ \mu\r_{ 1, \zeta } $:
\begin{equation*}
\mu\r_{ 1, \zeta }( \eta + \zeta, \lambda ) = 1 + \iint\limits_{ \mathbb{C} } g\r( \eta - q, \lambda ) v( q ) \mu\r_{ 1, \zeta }( q + \zeta, \lambda ) d \Re q d \Im q,
\end{equation*}
from which we obtain that $ \mu\r_{ 1, \zeta }( z, \lambda ) $ is defined for $ \lambda \in \mathbb{C} \backslash \mathcal{E}\r $ and is equal to $ \mu_{ 1, \zeta }( z, \lambda ) = \mu\r_1( z - \zeta, \lambda ) $. Now we substitute this formula for $ \mu\r_{ 1, \zeta } $ into the definition of $ a\r_{ 1, \zeta }( \lambda ) $:
\begin{equation*}
a\r_{ 1, \zeta }( \lambda ) = \iint\limits_{ \mathbb{C} } v_{ \zeta }( z ) \mu\r_{ 1, \zeta }( z, \lambda ) d \Re z d \Im z = \iint\limits_{ \mathbb{C} } v( z - \zeta ) \mu\r_1( z - \zeta, \lambda ) d \Re z d \Im z = a\r_1( \lambda ).
\end{equation*}
Similarly, for $ b\r_{ 1, \zeta }( \lambda ) $ we obtain
\begin{multline*}
b\r_{ 1, \zeta }( \lambda ) = \iint\limits_{ \mathbb{C} } e^{ i \lambda z + i \bar \lambda \bar z } v( z - \zeta ) \mu\r_1( z - \zeta, \lambda ) d \Re z d \Im z = \\
= \iint\limits_{ \mathbb{C} } e^{ i \lambda ( \eta + \zeta ) + i \bar \lambda ( \bar \eta + \bar \zeta ) } v( \eta ) \mu\r_1( \eta, \lambda ) d \Re \eta d \Im \eta = e^{ i \lambda \zeta + i \bar \lambda \bar \zeta } b\r_1( \lambda ).
\end{multline*}

Similarly formulas (\ref{c1_zeta}) and (\ref{d1_zeta}) can be obtained.

Function $ \mu\r_{ 2, \zeta }( z, \lambda ) $ corresponding to the potential $ v_{ \zeta }( z ) = v( z - \zeta ) $ is defined as the solution of the following equation
\begin{equation*}
\mu\r_{ 2, \zeta }( z, \lambda ) = z + \iint\limits_{ \mathbb{C} } g\r( z - \xi, \lambda ) v( \xi - \zeta ) \mu\r_{ 2, \zeta }( \xi, \lambda ) d \Re \xi d \Im \xi.
\end{equation*}
The change of variables $ \xi - \zeta = q $, $ z - \zeta = \eta $ yields the following equation on $ \mu\r_{ 2, \zeta } $:
\begin{equation*}
\mu\r_{ 2, \zeta }( \eta + \zeta, \lambda ) = \eta + \zeta + \iint\limits_{ \mathbb{C} } g\r( \eta - q, \lambda ) v( q ) \mu\r_{ 2, \zeta }( q + \zeta, \lambda ) d \Re q d \Im q,
\end{equation*}
from which we obtain that $ \mu\r_{ 2, \zeta }( z, \lambda ) $ is defined for $ \lambda \in \mathbb{C} \backslash \mathcal{E}\r $ and is equal to $ \mu\r_{ 2, \zeta }( z, \lambda ) = \mu\r_2( z - \zeta, \lambda ) + \zeta \mu\r_1( z - \zeta, \lambda ) $. Thus for the scattering data corresponding to $ \mu\r_{ 2, \zeta }( z, \lambda ) $ we have
\begin{multline*}
a\r_{ 2, \zeta }( \lambda ) = \iint\limits_{ \mathbb{C} } v_{ \zeta }( z ) \mu\r_{ 2, \zeta }( z, \lambda ) d \Re z d \Im z = \\
= \iint\limits_{ \mathbb{C} } v( z - \zeta ) \mu\r_2( z - \zeta, \lambda ) d \Re z d \Im z + \zeta \iint\limits_{ \mathbb{C} } v( z - \zeta ) \mu\r_1( z - \zeta ) d \Re z d \Im z = \\
= a\r_2( \lambda ) + \zeta a\r_1( \lambda ),
\end{multline*}

\begin{multline*}
c\r_{ 2, \zeta }( \lambda ) = \iint\limits_{ \mathbb{C} } z v_{ \zeta }( z ) \mu\r_{ 2, \zeta }( z, \lambda ) d \Re z d \Im z = \\
= \iint\limits_{ \mathbb{C} } z v( z - \zeta ) \mu\r_2( z - \zeta, \lambda ) d \Re z d \Im z + \zeta \iint\limits_{ \mathbb{C} } z v( z - \zeta ) \mu\r_1( z - \zeta, \lambda ) d \Re z d \Im z = \\
= c\r_{ 2 }( \lambda ) + \zeta a\r_2( \lambda ) + \zeta c\r_1( \lambda ) + \zeta^2 a\r_1( \lambda ).
\end{multline*}

Finally, function $ \mu\r_{ 3, \zeta }( z, \lambda ) $ corresponding to the potential $ v_{ \zeta }( z ) = v( z - \zeta ) $ is defined as the solution of the following equation
\begin{equation*}
\mu_{ 3, \zeta }( z, \lambda ) = z^2 + \iint\limits_{ \mathbb{C} } g\r( z - \xi, \lambda ) v( \xi - \zeta ) \mu\r_{ 3, \zeta }( \xi, \lambda ) d \Re \xi d \Im \xi.
\end{equation*}
The change of variables $ \xi - \zeta = q $, $ z - \zeta = \eta $ yields the following equation on $ \mu\r_{ 3, \zeta } $:
\begin{equation*}
\mu\r_{ 3, \zeta }( \eta + \zeta, \lambda ) = \eta^2 + 2 \eta \zeta + \zeta^2 + \iint\limits_{ \mathbb{C} } g\r( \eta - q, \lambda ) v( q ) \mu\r_{ 3, \zeta }( q + \eta, \lambda ) d \Re q d \Im q,
\end{equation*}
from which we obtain that $ \mu\r_{ 3, \zeta }( z, \lambda ) $ is defined for $ \lambda \in \mathbb{C} \backslash \mathcal{E}\r $ and is equal to $ \mu\r_{ 3, \zeta }( z, \lambda ) = \mu\r_3( z - \zeta, \lambda ) + 2 \zeta \mu\r_2( z - \zeta, \lambda ) + \zeta^2 \mu\r_1( z - \zeta, \lambda ) $. From this representation formula (\ref{a3_zeta}) can be easily obtained.

\end{enumerate}

\end{proof}

In order to prove Lemma \ref{dyn_lemma} we will need an auxiliary lemma.
\begin{lemma}
\label{aux_lemma}
Let $ T $ be an operator defined by
\begin{equation*}
T = \partial_{ t } + A = \partial_{ t } - 8 \partial_{ z }^3 - 2 w \partial_{ z } - 8 \partial_{ \bar z }^3 - 2 \bar w \partial_{ \bar z },
\end{equation*}
where $ w $ is defined via (\ref{w_def}), (\ref{w_decrease}) and it is assumed that $ v $ satisfies  conditions (\ref{smoothness})-(\ref{decrease}) and
\begin{equation*}
| \partial_{ t } v( x, t ) | \leqslant \frac{ \tilde q( t ) }{ ( 1 + | x | )^{ 4 + \varepsilon } }, \text{ for some } \tilde q( t ) > 0.
\end{equation*}

Suppose that for a certain $ \lambda \in \mathbb{C} \backslash 0 $, and every $ t \in \mathbb{R} $ solutions $ \mu\r_1( z, \lambda, t ) $, $ \mu\r_2( z, \lambda, t ) $, $ \mu\r_3( z, \lambda, t ) $ of (\ref{tilde_mu1}), (\ref{tilde_mu2}), (\ref{tilde_mu3}) correspondingly exist and are unique. Then
\begin{align}
\label{tpsi1_asympt}
& T \psi_1\r = & & \left( - 8 ( i \lambda )^3 + \frac{ 1 }{ 16 \pi } \left( - \mathcal{G} \partial_t a\r_{1} + 8 ( i \lambda )^3 \mathcal{G} a\r_1 \right) \right) e^{ i \lambda z } + \frac{ 1 }{ 16 \pi }\left( - \mathcal{G} \partial_t b\r_1 - 8 ( i \bar \lambda )^3 \mathcal{G} b\r_1 \right) e^{ - i \bar \lambda \bar z } + \\
\nonumber
&   & + & \overline{o}( 1 ) \text{ as } | z | \to \infty, \\
\label{tpsi2_asympt}
& T \psi\r_2 = & - & 8 ( i \lambda )^3 z e^{ i \lambda z } + \left( - 24 ( i \lambda )^2 + \frac{ 1 }{ 16 \pi }\left( - \mathcal{G} \partial_t a\r_{2} + 8 ( i \lambda )^3 \mathcal{G} a\r_2 \right) \right) e^{ i \lambda z } + \\
\nonumber
&   &+&\frac{ 1 }{ 16 \pi }\left( - \mathcal{G} \partial_t b\r_2 - 8 ( i \bar \lambda )^3 \mathcal{G} b\r_2 \right) e^{ - i \bar \lambda \bar z } + \overline{o}( 1 ), \quad \text{ as } | z | \to \infty, \\
\label{tpsi3_asympt}
& T \psi\r_3 = & - & 8 ( i \lambda )^3 z^2 e^{ i \lambda z } - 48 ( i \lambda )^2 z e^{ i \lambda z } + \left( - 48 ( i \lambda ) + \frac{ 1 }{ 16 \pi }\left( - \mathcal{G} \partial_t a\r_{3} + 8 ( i \lambda )^3 \mathcal{G} a\r_3 \right) \right) e^{ i \lambda z } + \\
\nonumber
&   &+& \frac{ 1 }{ 16 \pi }\left( - \mathcal{G} \partial_t b\r_3 - 8 ( i \bar \lambda )^3 \mathcal{G} b\r_3 \right) e^{ - i \bar \lambda \bar z } + \overline{o}( 1 ), \quad \text{ as } | z | \to \infty.
\end{align}
\end{lemma}

\begin{proof}[Proof of Lemma \ref{aux_lemma}]
We will only prove formula (\ref{tpsi1_asympt}). Formulas (\ref{tpsi2_asympt}), (\ref{tpsi3_asympt}) are proved similarly.

First of all, we note that
\begin{equation*}
\partial_t \mu\r_1 = \iint\limits_{ \mathbb{C} } g\r( z - \xi, \lambda ) \partial_t ( v( \xi, t ) \mu\r_1( \xi, \lambda, t ) ) d \Re \xi d \Im \xi.
\end{equation*}
From the assumptions of Lemma \ref{aux_lemma} it follows that the solution $ \frac{ \partial \mu\r_1 }{ \partial t } $ to this equation exists and is unique. From item \ref{g_lim} of Statement \ref{g_statement} it follows that $ \frac{ \partial \mu\r_1 }{ \partial t } = -\frac{ 1 }{ 16 \pi } \left( \mathcal{G} \partial_t a\r_1 + \mathcal{G} X \partial_t b\r_1 \right) + \overline{o}( 1 ) $ and
\begin{equation}
\label{partt_lim}
\frac{ \partial \psi\r_1 }{ \partial t } = -\frac{ 1 }{ 16 \pi } \mathcal{G} \partial_t a\r_1 e^{ i \lambda z } - \frac{ 1 }{ 16 \pi } \mathcal{G} \partial_t b\r_1 e^{ - i \bar \lambda \bar z } + \overline{o}( 1 ), \text{ as } | z | \to \infty.
\end{equation}

Further,
\begin{equation}
\label{part_psi_repres}
\partial_z^3 \psi\r_1 = ( i \lambda )^3 e^{ i \lambda z } \mu\r_1 + 3( i \lambda )^2 e^{ i \lambda z } \partial_z \mu\r_1 + 3 i \lambda e^{ i \lambda z } \partial_z^2 \mu\r_1 + e^{ i \lambda z } \partial_z^3 \mu\r_1.
\end{equation}

From item \ref{g_lim} of Statement \ref{g_statement} it follows that
\begin{equation}
\label{zero_asympt}
\mu\r_1 = 1 - \frac{ 1 }{ 16 \pi }( a\r_1( \lambda ) + b\r_1( \lambda ) X( z, \lambda ) ) \mathcal{G}( \lambda ) + \overline{o}( 1 ) \text{ as } | z | \to \infty.
\end{equation}

Function $ \partial_z \mu\r_1 $ is defined by
\begin{equation*}
\partial_z \mu\r_1 = \iint\limits_{ \mathbb{C} } \partial_z g\r( z - \xi, \lambda ) v( \xi, t ) \mu\r_1( \xi, \lambda, t ) d \Re \xi d \Im \xi,
\end{equation*}
where
\begin{equation*}
\partial_z g\r( z, \lambda ) = \frac{ - i \lambda }{ 16 \pi } e^{ - i \lambda z } ( \Ei( i \lambda z ) + \Ei( - i \bar \lambda \bar z ) ) + \frac{ 1 }{ 16 \pi z } - \frac{ ( - i \lambda ) }{ 16 \pi } X \mathcal{G}.
\end{equation*}
From item \ref{ei_est} of Statement \ref{ei_statement} it follows that $ \partial_z g\r = \frac{ i \lambda }{ 16 \pi } X \mathcal{G} + \underline{O}\left( \frac{ 1 }{ | z | } \right) $ as $ | z | \to \infty $ and thus
\begin{equation}
\label{one_repres}
\partial_z \mu\r_1 = \frac{ i \lambda }{ 16 \pi } X \mathcal{G} b\r_1 + \overline{o}( 1 ) \text{ as } | z | \to \infty.
\end{equation}

Function $ \partial^2_z \mu\r_1 $ is defined by
\begin{equation*}
\partial^2_z \mu\r_1 = \iint\limits_{ \mathbb{C} } \partial_z g\r( z - \xi, \lambda ) \partial_{ \xi } ( v( \xi, t ) \mu\r_1( \xi, \lambda, t ) ) d \Re \xi d \Im \xi.
\end{equation*}
Thus,
\begin{equation}
\label{two_repres}
\partial^2_z \mu\r_1 = \frac{ i \lambda }{ 16 \pi } X \mathcal{G} \iint\limits_{ \mathbb{C} } e^{ i \lambda \xi + i \bar \lambda \bar \xi }  \partial_{ \xi } ( v( \xi, t ) \mu\r_1( \xi, \lambda, t ) ) d \Re \xi d \Im \xi + \overline{o}( 1 ) = -\frac{ ( i \lambda )^2 }{ 16 \pi } X \mathcal{G} b\r_1 + \overline{o}( 1 ) \quad \text{ as } | z | \to \infty.
\end{equation}

Similarly it can be obtained that
\begin{equation}
\label{three_repres}
\partial^3_z \mu\r_1 = \frac{ ( i \lambda )^3 }{ 16 \pi } X \mathcal{G} b\r_1 + \overline{o}( 1 ) \quad \text{ as } | z | \to \infty.
\end{equation}
Thus (\ref{part_psi_repres})-(\ref{three_repres}) imply that
\begin{multline}
\label{partz_lim}
\partial_z^3 \psi\r_1 = \frac{ ( i \lambda )^3 }{ 16 \pi } ( 16 \pi - \mathcal{G} a\r_1 ) e^{ i \lambda z } - \frac{ ( i \lambda )^3 }{ 16 \pi } \mathcal{G} b\r_1 e^{ - i \bar \lambda \bar z } + 3 ( i \lambda )^2 \frac{ i \lambda }{ 16 \pi } \mathcal{G} b\r_1 e^{ - i \bar \lambda \bar z } - 3 i \lambda \frac{ ( - i \lambda )^2 }{ 16 \pi } \mathcal{G} b\r_1 e^{ - i \bar \lambda \bar z } - \frac{ ( - i \lambda )^3 }{ 16 \pi } \mathcal{G} b\r_1 e^{ - i \bar \lambda \bar z } + \\
+ \overline{o}( 1 ) = \frac{ ( i \lambda )^3 }{ 16 \pi } ( 16 \pi - \mathcal{G} a\r_1 ) e^{ i \lambda z } + \overline{o}( 1 ) \text{ as } | z | \to \infty.
\end{multline}

Further,
\begin{equation*}
\partial_{ \bar z }^3 \psi\r_1 = e^{ i \lambda z } \partial_{ \bar z }^3 \mu\r_1.
\end{equation*}
Function $ \partial_{ \bar z }^3 \mu\r_1 $ can be represented
\begin{equation*}
\partial_{ \bar z }^3 \mu\r_1 = \iint\limits_{ \mathcal{C} } \partial_{ \bar z } g\r( z - \xi, \lambda ) \partial_{ \bar \xi }^2 ( v( \xi, t ) \mu\r_1( \xi, \lambda, t ) ) d \Re \xi d \Im \xi,
\end{equation*}
where
\begin{equation*}
\partial_{ \bar z } g\r( z, \lambda ) = -\frac{ i \bar \lambda }{ 16 \pi \bar z } e^{ - i \lambda z - i \bar \lambda \bar z } + \frac{ i \bar \lambda }{ 16 \pi } X \mathcal{G}.
\end{equation*}

Thus
\begin{equation}
\label{partbarz_lim}
\partial_{ \bar z }^3 \psi\r_1 = \frac{ ( i \bar \lambda )^3 }{ 16 \pi } \mathcal{G} b\r_1 e^{ - i \bar \lambda \bar z } + \overline{o}( 1 ) \text{ as } | z | \to \infty.
\end{equation}

Combining (\ref{partt_lim}), (\ref{partz_lim}), (\ref{partbarz_lim}) and property (\ref{w_decrease}) of function $ w $ we obtain formula (\ref{tpsi1_asympt}).

\end{proof}

\begin{proof}[Proof of Lemma \ref{dyn_lemma}]
The derivation of formulas (\ref{a1_evolution})-(\ref{b1_evolution}), (\ref{d1_evolution})-(\ref{a2_evolution}) can be found in \cite{BLMP1}. We present here a slightly different approach applicable also to the derivation of formulas (\ref{c1_evolution}), (\ref{c2_evolution}).

Consider the following operator
\begin{equation*}
T = \partial_{t} - 8 \partial_{z}^3 - 2 w \partial_{ z }  - 8 \partial_{ \bar z }^3 - 2 \bar w \partial_{ \bar z },
\end{equation*}
where $ w $ is defined in (\ref{w_def}).

Equation (\ref{NV}) represents a condition under which the following is true
\begin{equation}
\label{equiv_formulation}
[ T, L ] \eta = 0, \quad \forall \eta \colon L \eta = 0,
\end{equation}
where $ L $ is defined in (\ref{schrodinger}) (see \cite{M}, \cite{BLMP1}). Note also that $ T = \partial_{ t } + A $, where $ A $ is the third order differential operator from the Manakov $ L-A-B $ triple for the two-dimensional Schr\"odinger operator $ L $ (see equations (\ref{t_schrodinger})-(\ref{A-B})).

Let us take $ \eta = \psi\r_1 $. Then (\ref{equiv_formulation}) is equivalent to
\begin{equation*}
L T \psi\r_1 = 0.
\end{equation*}
Thus function $ f_1 $ defined by
\begin{equation}
\label{f1_formula}
f_1( z, \lambda, t ) = T \psi\r_1( z, \lambda, t ),
\end{equation}
satisfies the Schr\"odinger equation (\ref{schrodinger}).

Formula (\ref{f1_formula}) can be rewritten
\begin{equation}
\label{partial_representation}
\partial_t \psi\r_1 = 8 \partial_z^3 \psi\r_1 + 2 w \partial_z \psi\r_1 + 8 \partial_{ \bar z }^3 \psi\r_1 + 2 \bar w \partial_{ \bar z } \psi\r_1 + f_1.
\end{equation}
Using this representation we can derive formulas for $ \partial_t a\r_1 $ and $ \partial_t b\r_1 $. Indeed, we write
\begin{equation}
\label{partiala_formula}
\partial_t a\r_1 = \iint\limits_{ \mathbb{C} } e^{ - i \lambda z } \partial_t v \psi\r_1 d \Re z d \Im z + \iint\limits_{ \mathbb{C} } e^{ - i \lambda z } v \partial_t \psi\r_1 d \Re z d \Im z.
\end{equation}
and
\begin{equation}
\label{partialb_formula}
\partial_t b\r_1 = \iint\limits_{ \mathbb{C} } e^{ i \bar \lambda \bar z } \partial_t v \psi\r_1 d \Re z d \Im z + \iint\limits_{ \mathbb{C} } e^{ i \bar \lambda \bar z } v \partial_t \psi\r_1 d \Re z d \Im z.
\end{equation}
Using (\ref{NV}), (\ref{partial_representation}), integrating expressions in (\ref{partiala_formula}), (\ref{partialb_formula}) by parts and using the fact that $ v \psi\r_1 = 4 \partial_z \partial_{ \bar z } \psi\r_1 $ we obtain
\begin{align}
\label{a1_repr}
& \partial_t a\r_1 = 8 ( i \lambda )^3 a\r_1 + \iint\limits_{ \mathbb{C} } e^{ - i \lambda z } v f_1 d \Re z d \Im z, \\
\label{b1_repr}
& \partial_t b\r_1 = - 8 ( i \bar \lambda )^3 b\r_1 + \iint\limits_{ \mathbb{C} } e^{ i \bar \lambda \bar z } v f_1 d \Re z d \Im z
\end{align}
(see subsection \ref{part_subs} of Appendix for details).

Inserting (\ref{a1_repr}), (\ref{b1_repr}) into (\ref{tpsi1_asympt}) and using (\ref{f1_formula}) we obtain that $ f_1 $ is a solution of the Schr\"odinger equation (\ref{schrodinger}) with the following asymptotics
\begin{multline*}
f_1( z, \lambda, t ) = \left( - 8 ( i \lambda )^3 - \frac{ 1 }{ 16 \pi } \mathcal{G} \iint\limits_{ \mathbb{C} } e^{ - i \lambda \xi } v f_1 d \Re \xi d \Im \xi \right) e^{ i \lambda z } + \\
+ \left( -\frac{ 1 }{ 16 \pi } \mathcal{G} \iint\limits_{ \mathbb{C} } e^{ i \bar \lambda \bar \xi } v f_1 d \Re \xi d \Im \xi \right) e^{ - i \bar \lambda \bar z } + \overline{o}( 1 ) , \text{ as } | z | \to \infty.
\end{multline*}
From the assumption of lemma that $ \lambda \not \in \mathcal{E}\r $ it follows that $ f_1 = - 8 ( i \lambda )^3 \psi\r_1 $. Using this fact formulas (\ref{c1_evolution}), (\ref{d1_evolution}) can be obtained similarly to the way formulas (\ref{a1_evolution}), (\ref{b1_evolution}) were derived (see subsection \ref{part_subs} of Appendix for details).

Now let us take $ \eta = \psi\r_2 $ in (\ref{equiv_formulation}). Then function $ f_2 $ defined by
\begin{equation*}
f_2( z, \lambda, t ) = T \psi\r_2 + 8 ( i \lambda )^3 \psi\r_2
\end{equation*}
is a solution of the Schr\"odinger equation (\ref{schrodinger}). Using this representation we can derive the following formulas for $ \partial_t a\r_2( \lambda, t ) $ and $ \partial_t b\r_2( \lambda, t ) $ ($ b\r_2 $ is defined by (\ref{b1_data}) with $ \mu\r_1 $ replaced by $ \mu\r_2 $):
\begin{align}
\label{a2_repr}
& \partial_t a\r_2 = 8 ( i \lambda )^3 a\r_2 - 8 ( i \lambda )^3 a\r_2 + \iint\limits_{ \mathbb{C} } e^{ - i \lambda z } v f_2 d \Re z d \Im z, \\
\label{b2_repr}
& \partial_t b\r_2 = -8 ( i \bar \lambda )^3 b\r_2 - 8 ( i \lambda )^3 b\r_2 + \iint\limits_{ \mathbb{C} } e^{ i \bar \lambda \bar z } v f_2 d \Re z d \Im z
\end{align}
(see subsection \ref{part_subs} of Appendix for details).
Inserting these representations into (\ref{tpsi2_asympt}) we obtain that
\begin{multline*}
f_2( z, \lambda, t ) = \left( - 24 ( i \lambda )^2 - \frac{ 1 }{ 16 \pi } \mathcal{G} \iint\limits_{ \mathbb{C} } e^{ - i \lambda \xi } v f_2 d \Re \xi d \Im \xi \right) e^{ i \lambda z } + \\
+ \left( -\frac{ 1 }{ 16 \pi } \mathcal{G} \iint\limits_{ \mathbb{C} } e^{ i \bar \lambda \bar \xi } v f_2 d \Re \xi d \Im \xi \right) e^{ - i \bar \lambda \bar z } + \overline{o}( 1 ) , \text{ as } | z | \to \infty.
\end{multline*}
Thus $ f_2( z, \lambda, t ) = - 24 ( i \lambda )^2 \psi\r_1 $. Using this result we can easily derive formulas (\ref{a2_evolution}) and (\ref{c2_evolution}) (see subsection \ref{part_subs} of Appendix for details).

Similarly let us take $ \eta = \psi\r_3 $ in (\ref{equiv_formulation}) and put
\begin{equation*}
f_3( z, \lambda, t ) = T \psi\r_3 + 8 ( i \lambda )^3 \psi\r_3 + 48 ( i \lambda )^2 \psi\r_2.
\end{equation*}
Using this representation we can derive the following formulas for $ \partial_t a\r_3( \lambda, t ) $ and $ \partial_t b\r_3( \lambda, t ) $ ($ b\r_3 $ is defined by (\ref{b1_data}) with $ \mu\r_1 $ replaced by $ \mu\r_3 $):
\begin{align*}
& \partial_t a\r_3 = 8 ( i \lambda )^3 a\r_3 - 8 ( i \lambda )^3 a\r_3 - 48 ( i \lambda )^2 a\r_2 + \iint\limits_{ \mathbb{C} } e^{ - i \lambda z } v f_3 d \Re z d \Im z, \\
& \partial_t b\r_3 = -8 ( i \bar \lambda )^3 b\r_3 - 8 ( i \lambda )^3 b\r_3 - 48 ( i \lambda )^2 b\r_2 + \iint\limits_{ \mathbb{C} } e^{ i \bar \lambda \bar z } v f_3 d \Re z d \Im z
\end{align*}
(see subsection \ref{part_subs} of Appendix for details).
Inserting these representations into (\ref{tpsi2_asympt}) we obtain that
\begin{multline*}
f_3( z, \lambda, t ) = \left( - 48 ( i \lambda )^2 - \frac{ 1 }{ 16 \pi } \mathcal{G} \iint\limits_{ \mathbb{C} } e^{ - i \lambda \xi } v f_2 d \Re \xi d \Im \xi \right) e^{ i \lambda z } + \\
+ \left( -\frac{ 1 }{ 16 \pi } \mathcal{G} \iint\limits_{ \mathbb{C} } e^{ i \bar \lambda \bar \xi } v f_2 d \Re \xi d \Im \xi \right) e^{ - i \bar \lambda \bar z } + \overline{o}( 1 ) , \text{ as } | z | \to \infty.
\end{multline*}
Besides, $ f_3 $ is a solution of the Schr\"odinger equation (\ref{schrodinger}). Thus, $ f_3( z, \lambda, t ) = - 48 ( i \lambda ) \psi\r_1 $. Using this result the formula (\ref{a3_evolution}) is easily derived (see subsection \ref{part_subs} of Appendix for details).

\end{proof}

\appendix
\section{Appendix}
\subsection{Proof of item \ref{ei_est} of Statement \ref{ei_statement}}
\label{ei_subs}
From property \ref{ei_int} we obtain that
\begin{equation*}
\Ei( z ) = \frac{ e^z }{ z } + \int\limits_{ -\infty }^{ z } \frac{ e^\tau }{ \tau^2 } d \tau, \quad \Ei( \bar z ) = \frac{ e^{ \bar z } }{ \bar z } + \int\limits_{ -\infty }^{ \bar z } \frac{ e^{ \tau } }{ \tau^2 } d \tau.
\end{equation*}
Evidently, $ \left| \frac{ e^{ -z } e^{ z } }{ z } \right| \leqslant \frac{ 1 }{ | z | } $, $ \left| \frac{ e^{ -z } e^{ \bar z } }{ \bar z } \right| \leqslant \frac{ 1 }{ | z | } $. Consider $ I_1 = \int\limits_{ -\infty }^{ z } \frac{ e^{ -z } e^{ \tau } }{ \tau^2 } d \tau $. Let us take the following contour of integration $ \Gamma_z = \Gamma_1 \cup \Gamma_2 $,
\begin{align*}
& \Gamma_1 = \{ \tau \in \mathbb{R}_-, -\infty < \tau < - | z | \}, \\
& \Gamma_2 = \begin{cases} & \{ \tau = | z | e^{ i \varphi }, \Arg z < \varphi < \pi \} \text{ if } \Arg z > 0, \\ & \{ \tau = | z | e^{ i \varphi }, - \pi < \varphi < \Arg z \} \text{ if } \Arg z < 0.
\end{cases}
\end{align*}
Note that $ \Re( - z + \tau ) < 0 $ on $ \Gamma_2 $. Then we can estimate
\begin{equation*}
\left| \, \int\limits_{ \Gamma_1 } \frac{ e^{ -z } e^{ \tau } }{ \tau^2 } d \tau \right| \leqslant \frac{ 1 }{ | z |^2 } \int\limits_{ -\infty }^{ -| z | } e^{ - \Re z + x } dx \leqslant \frac{ 1 }{ | z |^2 } \quad \text{ and } \quad \left| \, \int\limits_{ \Gamma_2 }\frac{ e^{ -z } e^{ \tau } }{ \tau^2 } d \tau \right| \leqslant \frac{ 2 \pi }{ | z | }.
\end{equation*}

Similarly, considering $ I_2 = \int\limits_{ -\infty }^{ \bar z } \frac{ e^{ -z } e^{ \tau } }{ \tau^2 } d \tau $ we take
\begin{align*}
& \Gamma_1 = \{ \tau \in \mathbb{R}_-, -\infty < \tau < - | z | \}, \\
& \Gamma_2 = \begin{cases} & \{ \tau = | z | e^{ i \varphi }, \Arg \bar z < \varphi < \pi \} \text{ if } \Arg \bar z > 0, \\ & \{ \tau = | z | e^{ i \varphi }, - \pi < \varphi < \Arg \bar z \} \text{ if } \Arg \bar z < 0.
\end{cases}
\end{align*}
and note that $ \Re( - z + \tau ) < 0 $ on $ \Gamma_2 $ which allows us to obtain that $ | I_2 | \leqslant \frac{ C }{ | z | } $ for $ | z | > 0 $.

\subsection{Behavior of $ H\r( \lambda ) $ at $ \lambda = 0 $ and at $ \lambda = \infty $}
\label{cont_subs}
Consider $ H\r( \lambda ) $ the integral operator of equations (\ref{m1_equation})-(\ref{m3_equation}). In Statement \ref{conth_statement} of the present subsection we prove that $ H\r( \lambda ) $ is continuous at $ \lambda = 0 $; in Statement \ref{vanh_statement} of the present subsection we prove that $ H\r( \lambda ) $ vanishes as $ | \lambda | \to \infty $. We use the fact that since $ H\r( \lambda ) $ is a Hilbert-Schmidt integral operator on $ L^2( \mathbb{C} ) $, its norm is estimated by $ \parallel H\r( \cdot, \cdot, \lambda ) \parallel_{ L^2( \mathbb{C} \times \mathbb{C} ) } $, where $ H\r( \cdot, \cdot, \lambda ) $ is the Schwartz kernel of the integral operator $ H\r( \lambda ) $.

\begin{statement}
\label{conth_statement}
\begin{equation*}
\iint\limits_{ \mathbb{C} } \iint\limits_{ \mathbb{C} }  \left| g\r( z - \xi, \lambda ) - \frac{ 1 }{ 16 \pi } \ln | z - \xi |^2 \right|^2 \frac{ | v ( \xi ) |^2 ( 1 + | \xi | )^{ 6 + \varepsilon } }{ ( 1 + | z | )^{ 6 + \varepsilon } }  d \Re z d \Im z d \Re \xi d \Im \xi  \to 0 \text{ as } | \lambda | \to 0.
\end{equation*}

\end{statement}
\begin{proof}
Let us find estimates on the difference $ g\r( z - \xi, \lambda ) - \frac{ 1 }{ 16 \pi } \ln| z - \xi |^2 $ for $ \lambda $ small enough.
\begin{enumerate}
\item $ | z - \xi | \leqslant 1 $

Using definitions (\ref{regg_def}) and (\ref{ei_definition}) we can estimate
\begin{multline}
\label{gdif1_estimate}
| g\r( z - \xi, \lambda ) - \frac{ 1 }{ 16 \pi } \ln| z - \xi |^2 | \leqslant \left| \frac{ 1 }{ 16 \pi } e^{ - i \lambda ( z - \xi ) } \ln| z - \xi |^2 - \frac{ 1 }{ 16 \pi } \ln| z - \xi |^2 \right| + \\
+ \left| \frac{ 1 }{ 4 \pi } e^{ - i \lambda ( z - \xi ) } ( 2 \gamma + \ln| \lambda |^2 ) - \frac{ 1 }{ 64 \pi } \left( 1 + e^{ - i \lambda ( z - \xi ) - i \bar \lambda ( \bar z - \bar \xi ) } \right) ( e^{ - i \lambda } + e^{ i \bar \lambda } ) ( 2 \gamma + \ln| \lambda |^2 ) \right| + \\
+ \Biggl| \sum_{ n = 1 }^{ \infty } \frac{ ( i \lambda ( z - \xi ) )^n + ( -i \bar \lambda ( \bar z - \bar \xi ) )^n }{ n n! } \cdot \frac{ 1 }{ 16 \pi } e^{ - i \lambda ( z - \xi ) } - \\
- \frac{ 1 }{ 64 \pi } \left( 1 + e^{ - i \lambda ( z - \xi ) - i \bar \lambda ( \bar z - \bar \xi ) } \right) ( e^{ - i \lambda } + e^{ i \bar \lambda } ) \sum_{ n = 1 }^{ \infty } \frac{ ( i \lambda )^n + ( -i \bar \lambda )^n }{ n n! }  \Biggr| \leqslant \\
\leqslant \const ( | \lambda | | z - \xi | | \ln | z - \xi |^2 | + \const | \lambda | | z - \xi | | \ln | \lambda |^2 | + | \lambda | ( 1 + | z - \xi | ) ).
\end{multline}

\item $ | z - \xi | > 1 $

\begin{enumerate}

\item $ | \lambda |^{ 1 - \theta } | z - \xi | \leqslant 1 $ for some $ 0 < \theta < 1 $

Using estimate (\ref{gdif1_estimate}) we obtain in this case
\begin{equation}
\label{gdif2_estimate}
| g\r( z - \xi, \lambda ) - \frac{ 1 }{ 16 \pi } \ln| z - \xi |^2 | \leqslant \const ( | \lambda |^\theta \ln | z - \xi |^2 + | \lambda |^{ \theta } | \ln | \lambda |^2 | + | \lambda |^\theta ).
\end{equation}

\item $ | \lambda |^{ 1 - \theta } | z - \xi | > 1 $ (note that the smaller $ \lambda $ is, the smaller this domain is in the $ \mathbb{C} \times \mathbb{C} $ space of $ ( z, \xi ) $ variables)

From property \ref{ei_est} of Statement \ref{ei_statement} we obtain
\begin{multline}
\label{gdif3_estimate}
| g\r( z - \xi, \lambda ) - \frac{ 1 }{ 16 \pi } \ln | z - \xi |^2 | \leqslant \const \left( \frac{ 1 }{ | \lambda | | z - \xi | } + \ln | z - \xi |^2 + | \ln | \lambda |^2 | \right) \leqslant \\
\leqslant \const ( | z - \xi |^{ \frac{ \theta }{ 1 - \theta } } + \ln | z - \xi |^2 + | \ln | z - \xi |^{ - \frac{ 2 }{ 1 - \theta } } | ).
\end{multline}

\end{enumerate}

\end{enumerate}
Estimates (\ref{gdif1_estimate}), (\ref{gdif2_estimate}) and (\ref{gdif3_estimate}) with $ \frac{ 2 \theta }{ 1 - \theta } < \varepsilon $, where $ \varepsilon $ is the constant from property (\ref{decrease}), yield the proof of the statement.
\end{proof}

\begin{statement}
\label{vanh_statement}
\begin{equation*}
\iint\limits_{ \mathbb{C} } \iint\limits_{ \mathbb{C} }  \left| g\r( z - \xi, \lambda ) \right|^2 \frac{ | v ( \xi ) |^2 ( 1 + | \xi | )^{ 6 + \varepsilon } }{ ( 1 + | z | )^{ 6 + \varepsilon } }  d \Re z d \Im z d \Re \xi d \Im \xi  \to 0 \text{ as } | \lambda | \to \infty.
\end{equation*}
\end{statement}

\begin{proof}
For definiteness we will assume that $ | \lambda | > 1 $.
\begin{enumerate}
\item $ | \lambda | | z - \xi | \geqslant 1  $ $ \Longrightarrow $ $ | \lambda |^{ \alpha }| z - \xi |^{ \alpha } \geqslant 1 $ $ \forall \alpha > 0 $

From item \ref{g_decr} of Statement \ref{g_statement} we obtain that
\begin{multline*}
\underset{ | \lambda || z - \xi | \geqslant 1 }{ \iint \iint } \left| g\r( z - \xi, \lambda ) \right|^2 \frac{ | v ( \xi ) |^2 ( 1 + | \xi | )^{ 6 + \varepsilon } }{ ( 1 + | z | )^{ 6 + \varepsilon } }  d \Re z d \Im z d \Re \xi d \Im \xi \leqslant \\
\leqslant \frac{ \const }{ | \lambda |^2 } \underset{ | \lambda || z - \xi | \geqslant 1 }{ \iint \iint } \left( 1 + \frac{ 1 }{ | z - \xi |^2 } \right) \frac{ | v( \xi ) |^2 ( 1 + | \xi | )^{ 6 + \varepsilon } }{ ( 1 + | z | )^{ 6 + \varepsilon } } d \Re z d \Im z d \Re \xi d \Im \xi \leqslant \\
\leqslant \const \frac{ \const }{ | \lambda |^{ 2 - 2 \alpha } } \iint\limits_{ \mathbb{C} }\iint\limits_{ \mathbb{C} } \left( 1 + \frac{ 1 }{ | z - \xi |^{ 2 - 2 \alpha } } \right) \frac{ | v( \xi ) |^2 ( 1 + | \xi | )^{ 6 + \varepsilon } }{ ( 1 + | z | )^{ 6 + \varepsilon } } d \Re z d \Im z d \Re \xi d \Im \xi \to 0 \text{ as } | \lambda | \to \infty
\end{multline*}
if $ \alpha < 1 $.

\item $ | \lambda || z - \xi | < 1 $

From item \ref{g_ext} of Statement \ref{g_statement} we obtain that
\begin{multline*}
\underset{ | \lambda | | z - \xi | < 1 }{ \iint \iint } \left| g\r( z - \xi, \lambda ) \right|^2 \frac{ | v ( \xi ) |^2 ( 1 + | \xi | )^{ 6 + \varepsilon } }{ ( 1 + | z | )^{ 6 + \varepsilon } }  d \Re z d \Im z d \Re \xi d \Im \xi \leqslant \\
\leqslant \const \ln | \lambda |^2 \underset{ | z - \xi | < 1 / | \lambda | }{ \iint \iint } \frac{ | v( \xi ) |^2 ( 1 + | \xi | )^{ 6 + \varepsilon } }{ ( 1 + | z | )^{ 6 + \varepsilon } } d \Re z d \Im z d \Re \xi d \Im \xi + \\
+ \const \underset{ | z - \xi | < 1 / | \lambda | }{ \iint \iint } \ln | z - \xi |^2 \frac{ | v( \xi ) |^2 ( 1 + | \xi | )^{ 6 + \varepsilon } }{ ( 1 + | z | )^{ 6 + \varepsilon } } d \Re z d \Im z d \Re \xi d \Im \xi.
\end{multline*}
The second summand in the last expression is $ \bar o( 1 ) $ as $ | \lambda | \to \infty $. For the first summand we get
\begin{multline*}
\const \ln | \lambda |^2 \underset{ | z - \xi | < 1 / | \lambda | }{ \iint \iint } \frac{ | v( \xi ) |^2 ( 1 + | \xi | )^{ 6 + \varepsilon } }{ ( 1 + | z | )^{ 6 + \varepsilon } } d \Re z d \Im z d \Re \xi d \Im \xi = \\
= \const \ln | \lambda |^2 \underset{ | w | < 1 / | \lambda | }{ \iint \iint } \frac{ 1 }{ ( 1 + | z - w | )^{ 2 + \varepsilon } ( 1 + | z | )^{ 6 + \varepsilon } } d \Re z d \Im z d \Re w d \Im w \leqslant \\
\leqslant \const \ln | \lambda |^2 \iint\limits_{ | w | < 1 / | \lambda |, } \; \iint\limits_{ | z | < 2 } \frac{ 1 }{ ( 1 + | z - w | )^{ 2 + \varepsilon } ( 1 + | z | )^{ 6 + \varepsilon } } d \Re z d \Im z d \Re w d \Im w + \\
+ \const \ln | \lambda |^2 \iint\limits_{ | w | < 1 / | \lambda |, } \; \iint\limits_{ | z | \geqslant 2 } \frac{ 1 }{ | z |^{ 2 + \varepsilon } ( 1 + | z | )^{ 6 + \varepsilon } } d \Re z d \Im z d \Re w d \Im w \leqslant \const \frac{ \ln| \lambda |^2 }{ | \lambda |^2 } \to 0 \text{ as } \lambda \to \infty.
\end{multline*}

\end{enumerate}
\end{proof}

\subsection{Derivation of formula (\ref{dif_det})}
\label{dif_subs}
Differentiating (\ref{fred_determinant}) with respect to $ \bar \lambda $ yields
\begin{equation*}
\frac{ \partial \ln \Delta\r }{ \partial \bar \lambda } = \Tr \left( - ( I - H\r( \lambda ) )^{ -1 } \frac{ \partial H\r }{ \partial \bar \lambda } + \frac{ \partial H\r }{ \partial \bar \lambda } \right).
\end{equation*}

From item \ref{g_der} of Statement \ref{g_statement} we obtain that
\begin{equation}
\label{first_part}
\Tr \frac{ \partial H\r }{ \partial \bar \lambda } = \left( \frac{ 1 }{ 4 \pi \bar \lambda } - \frac{ 1 }{ 2 \pi } \frac{ \partial \mathcal{G} }{ \partial \bar \lambda } \right) \hat v( 0 ).
\end{equation}

Now we note that due to item \ref{g_sym} of Statement \ref{g_statement} the following equation holds
\begin{multline*}
\overline {m\r_1}( z, \lambda ) X( z, \lambda ) = ( 1 + | z | )^{ - ( 3 + \varepsilon / 2 ) } X( z, \lambda ) + \\
+ \iint\limits_{ \mathbb{C} } ( 1 + | z | )^{ - ( 3 + \varepsilon / 2 ) } g\r( z - \xi, \lambda ) \frac{ v( \xi ) }{ ( 1 + | \xi | )^{ - ( 3 + \varepsilon / 2 ) } } \overline { m\r_1 }( \xi, \lambda ) X( \xi, \lambda ) d \Re \xi d \Im \xi,
\end{multline*}
which implies that
\begin{equation*}
( I - H\r( \lambda ) )^{ -1 } X( z, \lambda ) ( 1 + | z | )^{ - ( 3 + \varepsilon / 2 ) } = \overline{ m\r_1 }( z, \lambda ) X( z, \lambda )
\end{equation*}
for $ \lambda \in \mathbb{C} \backslash \mathcal{E}\r $.

Similarly we obtain that
\begin{equation*}
( I - H\r( \lambda ) )^{ -1 } X( z, \lambda ) \bar z ( 1 + | z | )^{ - ( 3 + \varepsilon / 2 ) } = \overline{ m\r_2 }( z, \lambda ) X( z, \lambda ).
\end{equation*}
Thus we obtain the following formula
\begin{equation}
\label{sec_part}
- \Tr \left( ( I - H\r( \lambda ) )^{ -1 } \frac{ \partial H\r }{ \partial \bar \lambda } \right) = - \frac{ 1 }{ 16 \pi \bar \lambda } \bar a_1 - \frac{ i \mathcal{G} }{ 16 \pi } \bar a_2 + \frac{ i \mathcal{G} }{ 16 \pi } \bar c_1 + \frac{ 1 }{ 16 \pi } \frac{ \partial \mathcal{G} }{ \partial \bar \lambda } a_1 + \frac{ 1 }{ 16 \pi } \frac{ \partial \mathcal{G} }{ \partial \bar \lambda } \bar a_1.
\end{equation}
Finally, combining (\ref{first_part}) with (\ref{sec_part}) we obtain (\ref{dif_det}).

The above derivation is rigorous for potentials $ v $ small enough. If we have an arbitrary potential $ v $, a similar formula can be derived for a potential $ \delta \cdot v $, where $ \delta \in \mathbb{C} $ is a small parameter. Since both parts of (\ref{dif_det}) are holomorphic with respect to $ \delta $ (see \cite{GK} for the proof of holomorphic dependence of $ \Delta $ on $ \delta $), then (\ref{dif_det}) holds for arbitrary values of $ \delta \in \mathbb{C} $.

\subsection{Derivation of formulas for $ \partial_t a\r_j $, $ \partial_t b\r_j $, $ \partial_t c\r_j $, $ \partial_t d\r_j $}
\label{part_subs}
We start by deriving a formula for $ \partial_t a\r_j $, $ j = 1, 2, 3 $. Substituting (\ref{NV}) and
\begin{equation*}
\partial_t \psi\r_j = 8 \partial_z^3 \psi\r_j + 2 w \partial_z \psi\r_j + 8 \partial_{ \bar z }^3 \psi\r_j + 2 \bar w \partial_{ \bar z } \psi\r_j + f
\end{equation*}
into
\begin{equation*}
\partial_t a\r_j = \iint\limits_{ \mathbb{C} } e^{ - i \lambda z } \partial_t v \psi\r_j d \Re z d \Im z + \iint\limits_{ \mathbb{C} } e^{ - i \lambda z } v \partial_t \psi\r_j d \Re z d \Im z
\end{equation*}
yields
\begin{multline}
\label{sum_eleven_a1}
\partial_{ t } a\r_j( \lambda, t ) = 8 \iint\limits_{ \mathbb{C} } e^{ - i \lambda z } \, \partial_{ z }^3 v \, \psi\r_j \,  d \Re z \, d \Im z + 8 \iint\limits_{ \mathbb{C} } e^{ - i \lambda z } \, \partial_{ \bar z }^3 v \, \psi\r_j \, d \Re z \, d \Im z + \\
+ 2 \iint\limits_{ \mathbb{C} } e^{ - i \lambda z } \, \partial_{ z } v \, w \, \psi\r_j \, d \Re z \, d \Im z + 2 \iint\limits_{ \mathbb{C} } e^{ - i \lambda z }\, v \, \partial_{ z } w \, \psi\r_j \, d \Re z \, d \Im z + 2 \iint\limits_{ \mathbb{C} } e^{ - i \lambda z } \, \partial_{ \bar z } v \, \bar w \, \psi\r_j \, d \Re z \, d \Im z + \\
+ 2 \iint\limits_{ \mathbb{C} } e^{ - i \lambda z } \, v \, \partial_{ \bar z } \bar w \, \psi\r_j \, d \Re z \, d \Im z + 8 \iint\limits_{ \mathbb{C} } e^{ - i \lambda z } \, v \partial_{ z }^{ 3 } \psi\r_j \, d \Re z \, d \Im z + 8 \iint\limits_{ \mathbb{C} } e^{ - i \lambda z } \, v \, \partial_{ \bar z }^{ 3 } \psi\r_j \, d \Re z \, d \Im z + \\
+ 2 \iint\limits_{ \mathbb{C} } e^{ - i \lambda z } \, v \, w \, \partial_{ z } \psi\r_j \, d \Re z \, d \Im z + 2 \iint\limits_{ \mathbb{C} } e^{ - i \lambda z } \, v \, \bar w \, \partial_{ \bar z } \psi\r_j \, d \Re z \, d \Im z + \iint\limits_{ \mathbb{C} } e^{ - i \lambda z } \, v \, f \, d \Re z \, d \Im z = \sum\limits_{ i = 1 }^{ 11 } I_i.
\end{multline}
Integrating $ I_7 $ by parts yields
\begin{multline*}
I_7 = - 8 ( -i \lambda )^3 \iint\limits_{ \mathbb{C} } e^{ - i \lambda z } \, v \, \psi\r_j \, d \Re z \, d \Im z - 24 ( - i \lambda )^2 \iint\limits_{ \mathbb{C} } e^{ - i \lambda z } \, \partial_{ z } v \, \psi\r_j \, d \Re z \, d \Im z - \\
- 24 ( - i \lambda ) \iint\limits_{ \mathbb{C} } e^{ - i \lambda z } \, \partial_{ z }^2 v \, \psi\r_j \, d \Re z \, d \Im z - 8 \iint\limits_{ \mathbb{C} } e^{ - i \lambda z } \, \partial_{ z }^3 v \, \psi\r_j \, d \Re z \, d \Im z.
\end{multline*}
In this way it can be obtained that
\begin{multline}
\label{first_sum_a1}
I_1 + I_2 + I_7 + I_{ 8 } = - 8 ( -i \lambda )^3 a\r_j - 24 ( - i \lambda )^2 \iint\limits_{ \mathbb{C} } e^{ - i \lambda z } \, \partial_{ z } v \, \psi\r_j \, d \Re z \, d \Im z - \\
- 24 ( - i \lambda ) \iint\limits_{ \mathbb{C} } e^{ - i \lambda z } \, \partial_{ z }^2 v \, \psi\r_j \, d \Re z \, d \Im z.
\end{multline}

Integrating $ I_{ 9 } $ by parts  and taking into account that $ - 4 \partial_{ z } \partial_{ \bar z } \psi\r_j + v \psi\r_j = 0 $ we obtain
\begin{multline*}
I_{ 9 } = - 2 \iint\limits_{ \mathbb{C} } e^{ - i \lambda z } \, \partial_{ z } v \, w \, \psi\r_j \, d \Re z \, d \Im z - 2 ( - i \lambda ) \iint\limits_{ \mathbb{C} } e^{ - i \lambda z } \, v \, w \, \psi\r_j \, d \Re z \, d \Im z - \\
- 2 \iint\limits_{ \mathbb{C} } e^{ - i \lambda z } \, v \, \partial_{ z } w \, \psi\r_j \, d \Re z \, d \Im z = - 2 \iint\limits_{ \mathbb{C} } e^{ - i \lambda z } \, \partial_{ z } v \, w \, \psi\r_j \, d \Re z \, d \Im z - \\
- 8 ( - i \lambda ) \iint\limits_{ \mathbb{C} } e^{ - i \lambda z } \, w \, \partial_{ z } \partial_{ \bar z } \psi\r_j \, d \Re z \, d \Im z  - 2 \iint\limits_{ \mathbb{C} } e^{ - i \lambda z } \, v \, \partial_{ z } w \, \psi\r_j \, d \Re z \, d \Im z = \\
= - 2 \iint\limits_{ \mathbb{C} } e^{ - i \lambda z } \, \partial_{ z } v \, w \, \psi\r_j \, d \Re z \, d \Im z
+ 24 ( - i \lambda )^2 \iint\limits_{ \mathbb{C} } e^{ - i \lambda z } \, \partial_{ z } v \, \psi\r_j \, d \Re z \, d \Im z - \\
+ 24 ( - i \lambda ) \iint\limits_{ \mathbb{C} } e^{ - i \lambda z } \, \partial^2_{ z } v \, \psi\r_j \, d \Re z \, d \Im z
- 2 \iint\limits_{ \mathbb{C} } e^{ - i \lambda z } \, v \, \partial_{ z } w \, \psi\r_j \, d \Re z \, d \Im z.
\end{multline*}

Thus it can be obtained that
\begin{multline}
\label{second_sum_a1}
I_{ 3 } + I_{ 4 } + I_{ 5 } + I_{ 6 } + I_{ 9 } + I_{ 10 } = \\
= 24 ( - i \lambda )^2 \iint\limits_{ \mathbb{C} } e^{ - i \lambda z } \, \partial_{ z } v \, \psi\r_j \, d \Re z \, d \Im z
+ 24 ( - i \lambda ) \iint\limits_{ \mathbb{C} } e^{ - i \lambda z } \, \partial_{ z }^2 v \, \psi\r_j \, d \Re z \, d \Im z.
\end{multline}
Finally,
\begin{equation*}
\partial_t a\r_j( \lambda, t ) = 8 ( i \lambda )^3 a\r_j( \lambda, t ) + \iint\limits_{ \mathbb{C} } e^{ - i \lambda z } v( z, t ) f( z, \lambda, t ) d \Re z d \Im z.
\end{equation*}

In order to derive a similar formula for $ \partial_t b\r_j $, $ j = 1, 2, 3 $, we need to replace $ i \lambda $ by $ - i \bar \lambda $ and $ - i \lambda z $ by $ i \bar \lambda \bar z $ everywhere in the derivation of the formula for $ \partial_t a\r_j $.

Next we derive a formula for $ \partial_t c\r_j $, $ j = 1, 2 $. Substituting (\ref{NV}) and
\begin{equation*}
\partial_t \psi\r_j = 8 \partial_z^3 \psi\r_j + 2 w \partial_z \psi\r_j + 8 \partial_{ \bar z }^3 \psi\r_j + 2 \bar w \partial_{ \bar z } \psi\r_j + f
\end{equation*}
into
\begin{equation*}
\partial_t c\r_j = \iint\limits_{ \mathbb{C} } z e^{ - i \lambda z } \partial_t v \psi\r_j d \Re z d \Im z + \iint\limits_{ \mathbb{C} } z e^{ - i \lambda z } v \partial_t \psi\r_j d \Re z d \Im z
\end{equation*}
yields
\begin{multline}
\label{sum_eleven_a1}
\partial_{ t } c\r_j( \lambda, t ) = 8 \iint\limits_{ \mathbb{C} } z e^{ - i \lambda z } \, \partial_{ z }^3 v \, \psi\r_j \,  d \Re z \, d \Im z + 8 \iint\limits_{ \mathbb{C} } z e^{ - i \lambda z } \, \partial_{ \bar z }^3 v \, \psi\r_j \, d \Re z \, d \Im z + \\
+ 2 \iint\limits_{ \mathbb{C} } z e^{ - i \lambda z } \, \partial_{ z } v \, w \, \psi\r_j \, d \Re z \, d \Im z + 2 \iint\limits_{ \mathbb{C} } z e^{ - i \lambda z }\, v \, \partial_{ z } w \, \psi\r_j \, d \Re z \, d \Im z + 2 \iint\limits_{ \mathbb{C} } z e^{ - i \lambda z } \, \partial_{ \bar z } v \, \bar w \, \psi\r_j \, d \Re z \, d \Im z + \\
+ 2 \iint\limits_{ \mathbb{C} } z e^{ - i \lambda z } \, v \, \partial_{ \bar z } \bar w \, \psi\r_j \, d \Re z \, d \Im z + 8 \iint\limits_{ \mathbb{C} } z e^{ - i \lambda z } \, v \partial_{ z }^{ 3 } \psi\r_j \, d \Re z \, d \Im z + 8 \iint\limits_{ \mathbb{C} } z e^{ - i \lambda z } \, v \, \partial_{ \bar z }^{ 3 } \psi\r_j \, d \Re z \, d \Im z + \\
+ 2 \iint\limits_{ \mathbb{C} } z e^{ - i \lambda z } \, v \, w \, \partial_{ z } \psi\r_j \, d \Re z \, d \Im z + 2 \iint\limits_{ \mathbb{C} } z e^{ - i \lambda z} \, v \, \bar w \, \partial_{ \bar z } \psi\r_j \, d \Re z \, d \Im z + \iint\limits_{ \mathbb{C} } z e^{ - i \lambda z } \, v \, f \, d \Re z \, d \Im z = \sum\limits_{ i = 1 }^{ 11 } J_i.
\end{multline}
Integrating $ J_7 $ by parts yields
\begin{multline*}
J_7 = - 8 ( -i \lambda )^3 \iint\limits_{ \mathbb{C} } z e^{ - i \lambda z } \, v \, \psi\r_j \, d \Re z \, d \Im z - 24 ( - i \lambda )^2 \iint\limits_{ \mathbb{C} } z e^{ - i \lambda z } \, \partial_{ z } v \, \psi\r_j \, d \Re z \, d \Im z - \\
- 24 ( - i \lambda )^2 \iint\limits_{ \mathbb{C} } e^{ - i \lambda z } \, v \, \psi\r_j \, d \Re z \, d \Im z - 24 ( - i \lambda ) \iint\limits_{ \mathbb{C} } z e^{ - i \lambda z } \, \partial_{ z }^2 v \, \psi\r_j \, d \Re z \, d \Im z - \\
- 48 ( - i \lambda ) \iint\limits_{ \mathbb{C} } e^{ - i \lambda z } \, \partial_{ z } v \, \psi\r_j \, d \Re z \, d \Im z - 8 \iint\limits_{ \mathbb{C} } z e^{ - i \lambda z } \, \partial_{ z }^3 v \, \psi\r_j \, d \Re z \, d \Im z - \\
- 24 \iint\limits_{ \mathbb{C} } e^{ - i \lambda z } \, \partial_{ z }^2 v \, \psi\r_j \, d \Re z \, d \Im z.
\end{multline*}
In this way it can be obtained that
\begin{multline}
\label{first_sum_a1}
J_1 + J_2 + J_7 + J_{ 8 } = - 8 ( -i \lambda )^3 c\r_j - 24 ( - i \lambda )^2 \iint\limits_{ \mathbb{C} } z e^{ - i \lambda z } \, \partial_{ z } v \, \psi\r_j \, d \Re z \, d \Im z - \\
- 24 ( - i \lambda )^2 \iint\limits_{ \mathbb{C} } e^{ - i \lambda z } \, v \, \psi\r_j \, d \Re z \, d \Im z - 24 ( - i \lambda ) \iint\limits_{ \mathbb{C} } z e^{ - i \lambda z } \, \partial_{ z }^2 v \, \psi\r_j \, d \Re z \, d \Im z - \\
- 48 ( - i \lambda ) \iint\limits_{ \mathbb{C} } e^{ - i \lambda z } \, \partial_{ z } v \, \psi\r_j \, d \Re z \, d \Im z - 24 \iint\limits_{ \mathbb{C} } e^{ - i \lambda z } \, \partial_{ z }^2 v \, \psi\r_j \, d \Re z \, d \Im z.
\end{multline}

Integrating $ J_{ 9 } $ by parts  and taking into account that $ - 4 \partial_{ z } \partial_{ \bar z } \psi\r_j + v \psi\r_j = 0 $ we obtain
\begin{multline*}
J_{ 9 } = - 2 \iint\limits_{ \mathbb{C} } z e^{ - i \lambda z } \, \partial_{ z } v \, w \, \psi\r_j \, d \Re z \, d \Im z - 2 ( - i \lambda ) \iint\limits_{ \mathbb{C} } z e^{ - i \lambda z } \, v \, w \, \psi\r_j \, d \Re z \, d \Im z  - \\
- 2 \iint\limits_{ \mathbb{C} } e^{ - i \lambda z } \, v \, w \, \psi\r_j \, d \Re z \, d \Im z
- 2 \iint\limits_{ \mathbb{C} } z e^{ - i \lambda z } \, v \, \partial_{ z } w \, \psi\r_j \, d \Re z \, d \Im z = \\
= - 2 \iint\limits_{ \mathbb{C} } z e^{ - i \lambda z } \, \partial_{ z } v \, w \, \psi\r_j \, d \Re z \, d \Im z
- 8 ( - i \lambda ) \iint\limits_{ \mathbb{C} } z e^{ - i \lambda z } \, w \, \partial_{ z } \partial_{ \bar z } \psi\r_j \, d \Re z \, d \Im z - \\
- 8 \iint\limits_{ \mathbb{C} } e^{ - i \lambda z } \, w \, \partial_{ z } \partial_{ \bar z } \psi\r_j \, d \Re z \, d \Im z
 - 2 \iint\limits_{ \mathbb{C} } z e^{ - i \lambda z } \, v \, \partial_{ z } w \, \psi\r_j \, d \Re z \, d \Im z = \\
= - 2 \iint\limits_{ \mathbb{C} } z e^{ - i \lambda z } \, \partial_{ z } v \, w \, \psi\r_j \, d \Re z \, d \Im z
+ 24 ( - i \lambda )^2 \iint\limits_{ \mathbb{C} } z e^{ - i \lambda z } \, \partial_{ z } v \, \psi\r_j \, d \Re z \, d \Im z + \\
+ 48 ( - i \lambda ) \iint\limits_{ \mathbb{C} } e^{ - i \lambda z } \, \partial_{ z } v \, \psi\r_j \, d \Re z \, d \Im z
+ 24 ( - i \lambda ) \iint\limits_{ \mathbb{C} } z e^{ - i \lambda z } \, \partial^2_{ z } v \, \psi\r_j \, d \Re z \, d \Im z + \\
+ 24 \iint\limits_{ \mathbb{C} } e^{ - i \lambda z } \, \partial^2_{ z } v \, \psi\r_j \, d \Re z \, d \Im z
- 2 \iint\limits_{ \mathbb{C} } z e^{ - i \lambda z } \, v \, \partial_{ z } w \, \psi\r_j \, d \Re z \, d \Im z.
\end{multline*}

Thus it can be obtained that
\begin{multline}
\label{second_sum_a1}
J_{ 3 } + J_{ 4 } + J_{ 5 } + J_{ 6 } + J_{ 9 } + J_{ 10 } = \\
= 24 ( - i \lambda )^2 \iint\limits_{ \mathbb{C} } z e^{ - i \lambda z } \, \partial_{ z } v \, \psi\r_j \, d \Re z \, d \Im z
+ 24 ( - i \lambda ) \iint\limits_{ \mathbb{C} } z e^{ - i \lambda z } \, \partial_{ z }^2 v \, \psi\r_j \, d \Re z \, d \Im z + \\
+ 48 ( - i \lambda ) \iint\limits_{ \mathbb{C} } e^{ - i \lambda z } \, \partial_{ z } v \, \psi\r_j \, d \Re z \, d \Im z +
24 \iint\limits_{ \mathbb{C} } e^{ - i \lambda z } \, \partial^2_{ z } v \, \psi\r_j \, d \Re z \, d \Im z.
\end{multline}
Finally,
\begin{equation*}
\partial_t c\r_j( \lambda, t ) = 8 ( i \lambda )^3 c\r_j( \lambda, t ) - 24( - i \lambda )^2 a\r_j + \iint\limits_{ \mathbb{C} } z e^{ - i \lambda z } v( z, t ) f( z, \lambda, t ) d \Re z d \Im z.
\end{equation*}

In order to derive a similar formula for $ \partial_t d\r_1 $ we need to replace $ i \lambda $ by $ - i \bar \lambda $, $ - i \lambda z $ by $ i \bar \lambda \bar z$ and $ z $ by $ \bar z $ (whenever it appears as a multiplier in the integrand) everywhere in the derivation of the formula for $ \partial_t c\r_1 $.

\end{document}